\documentclass[11pt,reqno,twoside]{amsart}
\usepackage{amssymb,amsmath,amsthm,soul,color,paralist}
\usepackage{t1enc}
\usepackage[cp1250]{inputenc}
\usepackage{a4,indentfirst,latexsym}
\usepackage{graphics}
\usepackage{mathrsfs}
\usepackage{cite,enumitem,graphicx}
\usepackage[colorlinks=true,urlcolor=blue,
citecolor=red,linkcolor=blue,linktocpage,pdfpagelabels,
bookmarksnumbered,bookmarksopen]{hyperref}
\usepackage[english]{babel}
\usepackage[left=2.50cm,right=2.50cm,top=2.72cm,bottom=2.72cm]{geometry}
\usepackage[metapost]{mfpic}
\usepackage[hyperpageref]{backref}
\usepackage[colorinlistoftodos]{todonotes}
\usepackage[normalem]{ulem}

\makeatletter
\providecommand\@dotsep{5}
\def\listtodoname{List of Todos}
\def\listoftodos{\@starttoc{tdo}\listtodoname}
\makeatother

\numberwithin{equation}{section}

\newcommand{\h}{H^{s}_{\e}}
\newcommand{\R}{\mathbb{R}}
\newcommand{\2}{2^{*}_{s}}

\newcommand{\C}{\mathbb{C}}

\newcommand{\N}{\mathcal{N}}

\DeclareMathOperator{\dive}{div}
\DeclareMathOperator{\supp}{supp}
\DeclareMathOperator{\e}{\varepsilon}

\newtheorem{prop}{Proposition}[section]
\newtheorem{lem}{Lemma}[section]
\newtheorem{thm}{Theorem}[section]

\newtheorem{cor}{Corollary}[section]

\newtheorem{remark}{Remark}[section]


\date{}

\begin{document}
\title[Critical fractional Schr\"odinger-Poisson equations with magnetic fields]{Multiplicity and concentration results for  fractional Schr\"odinger-Poisson equations with magnetic fields and critical growth}

\author[V. Ambrosio]{Vincenzo Ambrosio}
\address{Vincenzo Ambrosio\hfill\break\indent 
 Dipartimento di Scienze Matematiche, Informatiche e Fisiche \hfill\break\indent
 Universit\`a di Udine \hfill\break\indent
 via delle Scienze 206 \hfill\break\indent
 33100 Udine, Italy}
\email{vincenzo.ambrosio2@unina.it}

\begin{abstract}
We deal with the following fractional Schr\"odinger-Poisson equation with magnetic field
\begin{equation*}
\varepsilon^{2s}(-\Delta)_{A/\varepsilon}^{s}u+V(x)u+\e^{-2t}(|x|^{2t-3}*|u|^{2})u=f(|u|^{2})u+|u|^{\2-2}u \quad \mbox{ in } \mathbb{R}^{3},
\end{equation*}
where $\varepsilon>0$ is a small parameter, $s\in (\frac{3}{4}, 1)$, $t\in (0,1)$, $\2=\frac{6}{3-2s}$ is the fractional critical exponent, $(-\Delta)^{s}_{A}$ is the fractional magnetic Laplacian, $V:\R^{3}\rightarrow \R$ is a positive continuous potential, $A:\mathbb{R}^{3}\rightarrow \mathbb{R}^{3}$ is a smooth magnetic potential and $f:\mathbb{R}\rightarrow \mathbb{R}$ is a subcritical nonlinearity. 
Under a local condition on the potential $V$, we study the multiplicity and concentration of nontrivial solutions as $\e\rightarrow 0$.
In particular, we relate the number of nontrivial solutions with the topology of the set where the potential $V$ attains its minimum. 
\end{abstract}

\maketitle

\section{introduction}
In this paper we are concerned with the following fractional nonlinear Schr\"odinger-Poisson equation
\begin{equation}\label{P}
\varepsilon^{2s}(-\Delta)_{A/\varepsilon}^{s}u+V(x)u+\e^{-2t}(|x|^{2t-3}*|u|^{2})u=f(|u|^{2})u+|u|^{\2-2}u \quad \mbox{ in } \mathbb{R}^{3},
\end{equation}
where $\e>0$ is a parameter, $s\in (\frac{3}{4},1)$, $t\in (0, 1)$, $\2=\frac{6}{3-2s}$ is the fractional critical exponent, $V\in C(\R^{3}, \R)$ and $A\in C^{0,\alpha}(\R^{3},\R^{3})$, $\alpha\in(0,1]$,  are the electric and magnetic potentials respectively. Here the fractional magnetic Laplacian $(-\Delta)^{s}_{A}$ is defined, whenever $u\in C^{\infty}_{c}(\R^{3}, \C)$, as
\begin{equation}\label{operator}
(-\Delta)^{s}_{A}u(x)
:=c_{3,s} \lim_{r\rightarrow 0} \int_{B_{r}^{c}(x)} \frac{u(x)-e^{\imath (x-y)\cdot A(\frac{x+y}{2})} u(y)}{|x-y|^{3+2s}} dy,
\quad
c_{3,s}:=\frac{4^{s}\Gamma\left(\frac{3+2s}{2}\right)}{\pi^{3/2}|\Gamma(-s)|},
\end{equation}
and it has been recently considered in \cite{DS}.
The motivations that led to its introduction are mainly analyzed in \cite{DS, I10} and rely essentially on the L\'evy-Khintchine formula for the generator of a general L\'evy process.
As stated in \cite{SV}, this operator can be seen as the fractional counterpart of the magnetic Laplacian $-\Delta_{A}:=\left(\frac{1}{\imath}\nabla-A\right)^{2}$ given by
$$
-\Delta_{A} u= -\Delta u -\frac{2}{\imath} A(x) \cdot \nabla u + |A(x)|^{2} u -\frac{1}{\imath} u \dive(A(x));
$$ 
see \cite{LaL, LL, RS} for more details.
We recall that the magnetic Laplacian arises in the study of the following Schr\"odinger equation with magnetic field 
\begin{equation}\label{MSE}
-\Delta_{A} u+V(x)u=f(x, |u|^{2})u \quad \mbox{ in } \R^{N},
\end{equation}
for which a lot of interesting existence and multiplicity results have been established; see for instance \cite{AFF, AFY, AS, Cingolani, CS, DL, EL, K} and references therein.\\

In the nonlocal framework, only few and recent works deal with fractional magnetic Schr\"odinger equations like
\begin{equation}\label{MSE}
\e^{2s}(-\Delta)^{s}_{A}u+V(x)u=f(x, |u|^{2})u \quad \mbox{ in } \R^{N}.
\end{equation}
For instance, d'Avenia and Squassina \cite{DS} studied the existence of ground state to \eqref{MSE} when $\e=1$, $V$ is constant and $f$ is a subcritical or critical nonlinearity. Fiscella et al. \cite{FPV} proved the multiplicity of nontrivial solutions for a fractional magnetic problem with homogeneous boundary conditions. 
Zhang et al. \cite{ZSZ} obtained the existence of mountain pass solutions 
which tend to the trivial solution as $\e\rightarrow 0$ 
for a fractional magnetic Schr\"odinger equation involving critical frequency and critical growth.
In \cite{AD} the author and d'Avenia dealt with the existence and the multiplicity of solutions to \eqref{MSE} for small $\e>0$ when the potential $V$ satisfies \eqref{RVC} and $f$ has a subcritical growth. 

In absence of magnetic field (that is $A\equiv 0$), the fractional magnetic Laplacian $(-\Delta)^{s}_{A}$ reduces to the well-known fractional Laplacian $(-\Delta)^{s}$ which has achieved a tremendous popularity in these last twenty years due to its great applications in several contexts such as phase transitions, quasi-geostrophic flows, game theory, population dynamics, quantum mechanics and so on; see \cite{BuV, DPV, MBRS} for more details.
From a mathematical point of view,  several contributions \cite{A1, A0, A3, CN, FMV, FQT, Secchi1} have been given in the investigation of fractional Schr\"odinger equations like
\begin{equation}\label{FSE}
\e^{2s}(-\Delta)^{s}u+V(x)u=f(x,u) \mbox{ in } \R^{N},
\end{equation} 
which plays a crucial role in fractional quantum mechanics; see \cite{Laskin} and the appendix in \cite{DDPDV} for a more detailed physical interpretation.
In particular way, a special attention has been devoted to concentration phenomena of solutions to \eqref{FSE} as $\e\rightarrow 0$.
For instance, D\'avila et al. \cite{DDPW}, via a Lyapunov-Schmidt variational reduction, studied solutions to \eqref{FSE} with a spike pattern concentrating around a finite number of points in space as $\e\rightarrow 0$, when $V$ is a bounded sufficiently smooth potential and $f(u)=u^{p}$ with $p\in (1, \2-1)$.
Shang and Zhang \cite{SZ} dealt with the existence and multiplicity of solutions for a critical fractional Schr\"odinger equation requiring that the involved potential $V$ verifies the following condition due to Rabinowitz \cite{Rab}:
\begin{equation}\label{RVC}
\liminf_{|x|\rightarrow \infty} V(x)>\inf_{x\in \R^{N}} V(x).
\end{equation}
Dipierro et al. \cite{DMV} combined the Mountain Pass Theorem \cite{AR} and Concentration-Compactness Lemma to provide a multiplicity result for a fractional elliptic problem with critical growth. 
Alves and Miyagaki \cite{AM} (see also \cite{A1, A3, A-CPAA}) used a penalization argument to study the existence and concentration of positive solutions of \eqref{FSE} when $f$ has a subcritical growth and $V$ verifies the following assumptions due to del Pino and Felmer \cite{DF}:
\begin{compactenum}[$(V_1)$]
\item $\inf_{x\in \R^{3}} V(x)=V_{0}>0$;
\item  there exists a bounded domain $\Lambda\subset \R^{3}$ such that
\begin{equation}
V_{0}<\min_{\partial \Lambda} V \quad \mbox{ and } \quad M=\{x\in \Lambda: V(x)=V_{0}\}\neq \emptyset.
\end{equation}
\end{compactenum}
On the other hand, in these last years, some interesting papers appeared dealing with fractional Schr\"odinger-Poisson systems like
\begin{equation}\label{FSPS} 
\left\{
\begin{array}{ll}
\e^{2s}(-\Delta)^{s} u+V(x) \phi u = g(x, u) &\mbox{ in } \R^{3} \\
\e^{2t}(-\Delta)^{t} \phi= u^{2} &\mbox{ in } \R^{3},
\end{array}
\right.
\end{equation}
which can be considered as the nonlocal counterpart of the well-known Schr\"odinger-Poisson system which describes systems of identical charged particles interacting each other in the case that effects of magnetic field could be ignored and its solution represents, in particular, a standing wave for such a system; see \cite{BF}. In the classical case $s=t=1$, we refer to \cite{ASS, AdAP, ruiz, ZZ} and \cite{He, HL, WTXZ, Y} in which several results for unperturbed (i.e. $\e=1$)  and perturbed (i.e. $\e>0$ small) Schr\"odinger-Poisson systems and in absence of magnetic fields have been established, and \cite{BV, P, ZS} for some existence, uniqueness and multiplicity results when $A\not\equiv 0$.

Concerning \eqref{FSPS}, the first result is probably due to Giammetta \cite{G}, who studied the local and global well-posedness  of a fractional Schr\"odinger-Poisson system in which the fractional diffusion appears only in the second equation in \eqref{FSPS}.
In \cite{ZDS} Zhang et al.  used a perturbation approach to prove the existence of positive solutions to \eqref{FSPS} with $\e=1$, $V(x)=\mu>0$ and $g$ is a general nonlinearity having subcritical or critical growth.  
Murcia and Siciliano \cite{MS} showed that, for suitably small $\e$, the number of positive solutions to a doubly singularly perturbed fractional Schr\"odinger-Poisson system is estimated below by the Ljusternick-Schnirelmann category of the set of minima of the potential. 
Teng \cite{teng} investigated the existence of ground state solutions for a critical unperturbed fractional Schr\"odinger-Poisson system.
Liu and Zhang \cite{LZ} studied multiplicity and concentration of solutions to \eqref{FSPS} involving the fractional critical exponent and a potential $V$ satisfying global condition \eqref{RVC}.
To the best of our knowledge, fractional magnetic Schr\"odinger-Poisson equations like \eqref{P} have not ever been considered until now.
Particularly motivated by this fact and by the works \cite{AFF, AM, AD, LZ}, in the present paper we investigate the multiplicity and concentration of nontrivial solutions to \eqref{P} when $\e\rightarrow 0$, under assumptions $(V_1)$-$(V_2)$ on the continuous potential $V$, and 
 $f:\R\rightarrow \R$ is a $C^{1}$ function satisfying the following conditions:
\begin{compactenum}[$(f_1)$]
\item $f(t)=0$ for $t\leq 0$ and $\displaystyle{\lim_{t\rightarrow 0} \frac{f(t)}{t}=0}$;
\item there exist $q, \nu\in (4, 2^{*}_{s})$ and $\mu>0$ such that 
$$
f(t)\geq \mu t^{\frac{\nu-2}{2}} \quad \forall t>0 \, \mbox{ and } \lim_{t\rightarrow \infty} \frac{f(t)}{t^{\frac{q-2}{2}}}=0;
$$
\item there exists $\theta\in (4, q)$ such that $0<\frac{\theta}{2} F(t)\leq t f(t)$ for any $t>0$, where $F(t)=\int_{0}^{t} f(\tau)d\tau$;
\item $t\mapsto \frac{f(t)}{t}$ is increasing for $t>0$.
\end{compactenum} 
A typical example of function verifying $(f_1)$-$(f_4)$ is given by $f(t)=\sum_{i=1}^{k} \alpha_{i}(t^{+})^{\frac{q_{i}-2}{2}}$, with $\alpha_{i}\geq 0$ not all null and $q_{i}\in [\theta, 2^{*}_{s})$ for all $i\in \{1, \dots, k\}$.\\
Our main result can be stated as follows:
\begin{thm}\label{thm1}
Assume that $(V_1)$-$(V_2)$ and $(f_1)$-$(f_4)$ hold. Then, for any $\delta>0$ such that
$$
M_{\delta}=\{x\in \R^{3}: dist(x, M)\leq \delta\}\subset \Lambda,
$$ 
there exists $\e_{\delta}>0$ such that, for any $\e\in (0, \e_{\delta})$, problem \eqref{P} has at least $cat_{M_{\delta}}(M)$ nontrivial solutions. Moreover, if $u_{\e}$ denotes one of these solutions and $x_{\e}$ is a global maximum point of $|u_{\e}|$, then we have 
$$
\lim_{\e\rightarrow 0} V(x_{\e})=V_{0}
$$	
and
$$
|u_{\e}(x)|\leq \frac{C\e^{3+2s}}{C\e^{3+2s}+|x-x_{\e}|^{3+2s}} \quad \forall x\in \R^{3}.
$$
\end{thm}
\begin{remark}
Let us note that if $s, t \in (0, 1)$ are such that $4s+2t\geq 3$, then $H^{s}(\R^{3}, \R)\subset L^{\frac{12}{3+2t}}(\R^{3}, \R)$ and $\phi_{|u|}^{t}$ is well-defined; see Section \ref{sec2} below. Therefore, $s\in (\frac{3}{4}, 1)$ and $t\in (0, 1)$ are admissible exponents. Moreover, the restriction $s\in (\frac{3}{4}, 1)$ is related to the growth assumptions on $f$  (in fact, we have that $2^{*}_{s}>4$) which allow us to apply variational arguments, use the Nehari manifold and verify the Palais-Smale condition; see Sections \ref{sec3} and \ref{sec4} below. For what concerns the dimension $N=3$, we suspect that our results can be extended only in low dimensions such that $N\leq 4s+2t$ (see for instance \cite{MS}) and considering more general nonlinearities such that $\frac{F(u)}{u^{2}}\rightarrow \infty$ as $u\rightarrow \infty$ and that do not verify $(f_4)$. Anyway, the three dimensional case is relevant for the physical meaning of the fractional Schr\"odinger-Poisson system.
\end{remark}
The proof of Theorem \ref{thm1} relies on suitable variational methods and Ljusternik-Schnirelmann theory inspired by \cite{ADOS} and \cite{AFF} in which the authors dealt with classical Schr\"odinger equations with critical growth and $A\equiv 0$ and subcritical growth and $A\not\equiv 0$ respectively. First of all we note that, using the change of variable $x\mapsto \e x$, problem (\ref{P}) is equivalent to the following one
\begin{equation}\label{Pe}
(-\Delta)_{A_{\e}}^{s} u + V_{\e}( x)u +\phi_{|u|}^{t} u=  f(|u|^{2})u+|u|^{\2-2}u  \mbox{ in } \R^{3},
\end{equation}
where $A_{\e}(x)=A(\e x)$, $V_{\e}(x)=V(\e x)$ and $\phi_{|u|}^{t}=|x|^{2t-3}*|u|^{2}$. Since we do not have any information on the behavior of $V$ at infinity, we adapt the penalization argument developed by del Pino and Felmer in \cite{DF}, which consists in modifying the nonlinearity $f$ in a special way and to consider an auxiliary problem. More precisely, as in \cite{AFF}, we fix $k>\frac{\theta}{\theta-2}$ and $a>0$ such that $f(a)+a^{\frac{\2-2}{2}}=\frac{V_{0}}{k}$, and we consider the function
$$
\hat{f}(t):=
\begin{cases}
f(t)+(t^+)^{\frac{\2-2}{2}} & \text{ if $t \leq a$} \\
\frac{V_{0}}{k}    & \text{ if $t >a$}.
\end{cases}
$$ 
Let $t_{a}, T_{a}>0$ such that $t_{a}<a<T_{a}$ and take $\xi\in C^{\infty}_{c}(\R, \R)$ such that
\begin{compactenum}[$(\xi_1)$]
\item $\xi(t)\leq \hat{f}(t)$ for all $t\in [t_{a}, T_{a}]$,
\item $\xi(t_{a})=\hat{f}(t_{a})$, $\xi(T_{a})=\hat{f}(T_{a})$, $\xi'(t_{a})=\hat{f}'(t_{a})$ and $\xi'(T_{a})=\hat{f}'(T_{a})$, 
\item the map $t\mapsto \frac{\xi(t)}{t}$ is increasing for all $t\in [t_{a}, T_{a}]$.
\end{compactenum}
Then we define $\tilde{f}\in C^{1}(\R, \R)$ as follows:
$$
\tilde{f}(t):=
\begin{cases}
\hat{f}(t)& \text{ if $t\notin [t_{a}, T_{a}]$} \\
\xi(t)    & \text{ if $t\in [t_{a}, T_{a}]$}.
\end{cases}
$$ 
Finally, we introduce the following penalized nonlinearity $g: \R^{3}\times \R\rightarrow \R$ by setting
$$
g(x, t)=\chi_{\Lambda}(x)(f(t)+(t^+)^{\frac{\2-2}{2}})+(1-\chi_{\Lambda}(x))\tilde{f}(t),
$$
where $\chi_{\Lambda}$ is the characteristic function on $\Lambda$, and  we set $G(x, t)=\int_{0}^{t} g(x, \tau)\, d\tau$.
From assumptions $(f_1)$-$(f_4)$ and $(\xi_1)$-$(\xi_3)$, it follows that $g$ verifies the following properties:
\begin{compactenum}[($g_1$)]
\item $\displaystyle{\lim_{t\rightarrow 0} \frac{g(x, t)}{t}=0}$ uniformly in $x\in \R^{3}$;
\item $g(x, t)\leq f(t)+t^{\frac{\2-2}{2}}$ for any $x\in \R^{3}$ and $t>0$;
\item $(i)$ $0< \frac{\theta}{2} G(x, t)\leq g(x, t)t$ for any $x\in \Lambda$ and $t>0$, \\
$(ii)$ $0\leq  G(x, t)\leq g(x, t)t\leq \frac{V(x)}{k}t$ and $0\leq g(x,t)\leq \frac{V(x)}{k}$ for any $x\in \Lambda^{c}$ and $t>0$;
\item $t\mapsto \frac{g(x,t)}{t}$ is increasing for all $x\in \Lambda$ and $t>0$.
\end{compactenum}
Then we consider the following modified problem 
\begin{equation}\label{MPe}
(-\Delta)^{s}_{A_{\e}} u + V_{\e}(x)u +\phi_{|u|}^{t}u=  g_{\e}(x, |u|^{2})u \mbox{ in } \R^{3}, 
\end{equation}
where $g_{\e}(x,t):=g(\e x, t)$.
Let us note that if $u$ is a solution of (\ref{MPe}) such that 
\begin{equation}\label{ue}
|u(x)|\leq t_{a} \mbox{ for all } x\in  \Lambda_{\e}^{c},
\end{equation}
where $\Lambda_{\e}:=\{x\in \R^{3}: \e x\in \Lambda\}$, then $u$ is indeed a solution of the original problem  (\ref{Pe}).

Since we want to find nontrivial solutions  to \eqref{Pe}, we look for critical points of the following functional associated with \eqref{Pe}:
\begin{align*}
J_{\e}(u)&=\frac{c_{3,s}}{2}\iint_{\R^{6}} \frac{|u(x)-e^{\imath (x-y)\cdot A_{\e}(\frac{x+y}{2})} u(y)|^{2}}{|x-y|^{3+2s}} \, dxdy+\frac{1}{2}\int_{\R^{3}} V_{\e}(x) |u|^{2}\, dx\\
&\quad+\frac{1}{4}\int_{\R^{3}}\phi_{|u|}^{t}|u|^{2}dx -\frac{1}{2}\int_{\R^{3}} G_{\e}(x, |u|^{2})\, dx
\end{align*}
defined on the fractional Sobolev space 
$$
\h=\left\{u\in \mathcal{D}^{s}_{A_{\e}}(\R^{3}, \C): \int_{\R^{3}} V_{\e}(x) |u|^{2}\, dx<\infty\right\};
$$
see Section $2$ for more details.
The main difficulty in the study of $J_{\e}$ is related to verify a local Palais-Smale compactness condition at any level $c<c_{*}:=\frac{s}{3}S_{*}^{\frac{3}{2s}}$, where $S_{*}$ is the best Sobolev constant of the embedding $H^{s}(\R^{3}, \R)$ in $L^{\2}(\R^{3}, \R)$.
Indeed, the appearance of the magnetic field, the critical exponent, the convolution term $|x|^{2t-3}*|u|^{2}$ and the nonlocal nature of the fractional magnetic Laplacian, make our analysis much more complicated and delicate with respect to \cite{ADOS, AFF, AM, AD, LZ}. 
We circumvent these issues proving some careful estimates and using the Concentration-Compactness Lemma for the fractional Laplacian \cite{A-CPAA, DMV, PP}; see Lemma \ref{PSc}. The H\"older regularity assumption on the magnetic field $A$ and the fractional diamagnetic inequality established in \cite{DS} will be used to show that the mountain pass minimax level $c_{\e}$ of $J_{\e}$ is less than $c_{*}$ for $\e>0$ small enough.
In order to obtain multiple solutions for the modified problem, we use some techniques developed by Benci and Cerami in \cite{BC}, which are based on suitable comparisons between the category of some sublevel sets of the modified functional and the category of the set $M$.
After that, we need to prove that if $u_{\e}$ is a solution of modified problem \eqref{MPe}, then $|u_{\e}|$ satisfies \eqref{ue} for $\e$ small enough.
In order to achieve our goal, 
we aim to show that the (translated) sequence $(u_{n})$ verifies the property $|u_{n}(x)|\rightarrow 0$ as $|x|\rightarrow \infty$ uniformly with respect to $n\in \mathbb{N}$.
In the case $A=0$ (see for instance \cite{AM, A-CPAA}), this is proved by using some fundamental estimates established in \cite{FQT} concerning the Bessel operator. When $A\not\equiv 0$, we do not have similar informations for the following fractional equation
\begin{equation*}
(-\Delta)^{s}_{A}u+V_{0}u=h(|u|^{2})u \mbox{ in } \R^{3}.
\end{equation*}
To overcome this difficulty, we use a clever approximation argument which allows us to deduce that if $u$ is  a solution to \eqref{MPe}, then $|u|$ is a subsolution to 
\begin{equation*}
(-\Delta)^{s}u+V_{0}u=g_{\e}(x, |u|^{2})|u| \mbox{ in } \R^{3};
\end{equation*}
see Lemma \ref{moser}. We recall that in the case $s=1$, it is clear that if $u$ is a solution to 
$$
-\Delta_{A} u+V_{0}u=h(|u|^{2})u \mbox{ in } \R^{3},
$$
then $|u|$ is a subsolution to 
$$
-\Delta|u|+V_{0}|u|=h(|u|^{2})|u| \mbox{ in } \R^{3},
$$
in view of the Kato's inequality \cite{Kato} 
$$
-\Delta|u|\leq \Re(sign(u)(-\Delta_{A} u)),
$$
and then we can apply standard arguments to prove that $|u(x)|\rightarrow 0$ as $|x|\rightarrow \infty$ (the decay is exponential); see for instance \cite{K}.
Unfortunately, in our setting, 
even if we suspect that a distributional Kato's inequality for \eqref{operator} holds true (see for instance \cite{AD} in which a pointwise fractional magnetic Kato's inequality is used), we  are not able to prove it.
We point out that in \cite{HIL}, the authors obtained a Kato's inequality for magnetic relativistic Schr\"odinger operators $$H_{A, m}^{\beta}=[(-\imath \nabla-A(x))^{2}+m^{2}]^{\beta/2}$$ with $m\geq 0$ and $\beta\in (0, 1]$, which include  \eqref{operator} when $\beta=1$ and $m=0$, that is $H^{1}_{A, 0}=(-\Delta)^{1/2}_{A}$.
On the other hand, due to the nonlocal character of \eqref{operator}, we cannot adapt in our framework the arguments developed in \cite{AFF} to prove that $|u_{n}(x)|\rightarrow 0$ as $|x|\rightarrow \infty$ uniformly with respect to $n\in \mathbb{N}$. For the above reasons, in this work we develop some new ideas needed to achieve our claim. Roughly speaking, we will show that a Kato's inequality holds for the modified problem \eqref{MPe}.
More precisely, we first show that  each  $|u_{n}|$ is bounded in $L^{\infty}(\R^{3}, \R)$-norm uniformly in $n\in \mathbb{N}$, by means of a Moser iteration argument \cite{Moser}. At this point, we prove that each $|u_{n}|$ verifies
\begin{equation*}
(-\Delta)^{s}|u_{n}|+V_{0}|u_{n}|\leq g_{\e}(x, |u_{n}|^{2})|u_{n}| \mbox{ in } \R^{3},
\end{equation*}
by using $\displaystyle{\frac{u_{n}}{u_{\delta,n}}\varphi}$ as test function in the modified problem, 
where $u_{\delta,n}=\sqrt{|u_{n}|^{2}+\delta^{2}}$ and $\varphi$ is a real smooth nonnegative function with compact support in $\R^{3}$, and then we pass to the limit as $\delta\rightarrow 0$. 
This  fact combined with a comparison argument and the results in \cite{AM, FQT}, allows us to deduce that $|u_{n}(x)|\rightarrow 0$ as $|x|\rightarrow \infty$ uniformly with respect to $n\in \mathbb{N}$; see Lemma \ref{moser}.
Finally, we give a decay estimate of modulus $|u_{\e}|$ of solutions $u_{\e}$ to \eqref{P}. 

As far as we know,  this is the first time that penalization methods jointly with Ljusternik-Schnirelmann theory are used to obtain multiple solutions for a fractional magnetic Schr\"odinger-Poisson equation with critical growth. 
\noindent

The paper is structured as follows. In Section \ref{sec2} we recall some properties on the involved fractional Sobolev spaces. In Section \ref{sec3} we prove some compactness properties for the modified functional. In Section \ref{sec4} we introduce the barycenter map which will be a fundamental tool to obtain a multiplicity result for problem \eqref{MPe} via Ljusternick-Schnirelmann theory. In the last section we give the proof of Theorem \ref{thm1}.

\section{Preliminaries}\label{sec2}
In this section we collect some notations and technical lemmas which will be used along the paper.

We define $H^{s}(\R^{3}, \R)$ as the fractional Sobolev space 
$$
H^{s}(\R^{3}, \R)=\{u\in L^{2}(\R^{3}, \R): [u]<\infty\}
$$
where
$$
[u]^{2}=\iint_{\R^{6}} \frac{|u(x)-u(y)|^{2}}{|x-y|^{3+2s}} dxdy.
$$
We recall that the embedding $H^{s}(\R^{3}, \R)\subset L^{q}(\R^{3}, \R)$ is continuous for all $q\in [2, \2)$ and locally compact for all $q\in [1, \2)$; see \cite{DPV, MBRS} for more details on this topic.

Let $L^{2}(\R^{3}, \C)$ be the space of complex-valued functions such that $\int_{\R^{3}}|u|^{2}\, dx<\infty$ endowed with the inner product 
$\langle u, v\rangle_{L^{2}}=\Re\int_{\R^{3}} u\bar{v}\, dx$, where the bar denotes complex conjugation.

Let us denote by
$$
[u]^{2}_{A}:=\frac{c_{3,s}}{2} \iint_{\R^{6}} \frac{|u(x)-e^{\imath (x-y)\cdot A(\frac{x+y}{2})} u(y)|^{2}}{|x-y|^{3+2s}} \, dxdy,
$$
and consider
$$
D_A^s(\R^3,\C)
:=
\left\{
u\in L^{2_s^*}(\R^3,\C) : [u]^{2}_{A}<\infty
\right\}.
$$
Then we introduce the Hilbert space
$$
H^{s}_{\e}:=
\left\{
u\in D_{A_{\e}}^s(\R^3,\C): \int_{\R^{3}} V_{\e}(x) |u|^{2}\, dx <\infty
\right\}
$$ 
endowed with the scalar product
\begin{align*}
\langle u , v \rangle_{\e}&=
\Re	\int_{\R^{3}} V_{\e}(x) u \bar{v} dx\\
&
+ \frac{c_{3,s}}{2}\Re\iint_{\R^{6}} \frac{(u(x)-e^{\imath(x-y)\cdot A_{\e}(\frac{x+y}{2})} u(y))\overline{(v(x)-e^{\imath(x-y)\cdot A_{\e}(\frac{x+y}{2})}v(y))}}{|x-y|^{3+2s}} dx dy
\end{align*}
and let
$$
\|u\|_{\e}:=\sqrt{\langle u , u \rangle_{\e}}=\sqrt{[u]^{2}_{A_{\e}}+\| \sqrt{V_{\e}}|u| \|^{2}_{L^{2}(\R^{3})}}.
$$
The space $\h$ satisfies the following fundamental properties; see \cite{AD, DS} for more details.
\begin{lem}\cite{AD, DS}
The space $\h$ is complete and $C_c^\infty(\R^3,\C)$ is dense in $\h$. 
\end{lem}

\begin{lem}\label{DI}\cite{DS}
If $u\in H^{s}_{A}(\R^{3}, \C)$ then $|u|\in H^{s}(\R^{3}, \R)$ and we have
$$
[|u|]\leq [u]_{A}.
$$
\end{lem}

\begin{thm}\label{Sembedding}\cite{DS}
	The space $H^{s}_{\e}$ is continuously embedded in $L^{r}(\R^{3}, \C)$ for $r\in [2, 2^{*}_{s}]$, and compactly embedded in $L_{\rm loc}^{r}(\R^{3}, \C)$ for $r\in [1, 2^{*}_{s})$.\\
\end{thm}

\begin{lem}\label{aux}\cite{AD}
If $u\in H^{s}(\R^{3}, \R)$ and $u$ has compact support, then $w=e^{\imath A(0)\cdot x} u \in \h$.
\end{lem}

We also recall the following vanishing lemma \cite{FQT}:
\begin{lem}\label{Lions}\cite{FQT}
	Let $q\in [2, \2)$. If $(u_{n})$ is a bounded sequence in $H^{s}(\R^{3}, \R)$ and if 
	\begin{equation*}
	\lim_{n\rightarrow \infty} \sup_{y\in \R^{3}} \int_{B_{R}(y)} |u_{n}|^{q} dx=0
	\end{equation*}
	for some $R>0$, then $u_{n}\rightarrow 0$ in $L^{r}(\R^{3}, \R)$ for all $r\in (2, 2^{*}_{s})$.
\end{lem}

Now, let $s, t\in (0, 1)$ such that $4s+2t\geq 3$. Using the embedding $H^{s}(\R^{3}, \R)\subset L^{q}(\R^{3}, \R)$ for all $q\in [2, \2)$, we can see that 
\begin{equation}\label{ruiz}
H^{s}(\R^{3}, \R)\subset L^{\frac{12}{3+2t}}(\R^{3}, \R).
\end{equation}
For any $u\in \h$, we get $|u|\in H^{s}(\R^{3}, \R)$ by Lemma \ref{DI}, and the linear functional $\mathcal{L}_{|u|}: D^{t, 2}(\R^{3}, \R)\rightarrow \R$ given by
$$
\mathcal{L}_{|u|}(v)=\int_{\R^{3}} |u|^{2} v\, dx 
$$
is well defined and continuous in view of H\"older inequality and \eqref{ruiz}. 
Indeed, we can see that
\begin{equation}
|\mathcal{L}_{|u|}(v)|\leq \left(\int_{\R^{3}} |u|^{\frac{12}{3+2t}}dx\right)^{\frac{3+2t}{6}}  \left(\int_{\R^{3}} |v|^{2^{*}_{t}}dx\right)^{\frac{1}{2^{*}_{t}}}\leq C\|u\|^{2}_{D^{s,2}}\|v\|_{D^{t,2}},
\end{equation}
where
$$
\|v\|^{2}_{D^{t,2}}=\iint_{\R^{6}} \frac{|v(x)-v(y)|^{2}}{|x-y|^{3+2t}}dxdy.
$$
Then, by the Lax-Milgram Theorem there exists a unique $\phi_{|u|}^{t}\in D^{t, 2}(\R^{3}, \R)$ such that 
\begin{equation}\label{FPE}
(-\Delta)^{t} \phi_{|u|}^{t}= |u|^{2} \mbox{ in } \R^{3}.
\end{equation}
Therefore we obtain the following $t$-Riesz formula
\begin{equation}\label{RF}
\phi_{|u|}^{t}(x)=c_{t}\int_{\R^{3}} \frac{|u(y)|^{2}}{|x-y|^{3-2t}} \, dy  \quad (x\in \R^{3}), \quad c_{t}=\pi^{-\frac{3}{2}}2^{-2t}\frac{\Gamma(3-2t)}{\Gamma(t)}.
\end{equation}
We note that the above integral is convergent at infinity since $|u|^{2}\in L^{\frac{6}{3+2t}}(\R^{3}, \R)$.\\
In the sequel, we will omit the constants $c_{3,s}$ and $c_{t}$ in order to lighten the notation. We conclude this section giving some properties on the convolution term.
\begin{lem}\label{poisson}
Let us assume that $4s+2t\geq 3$ and $u\in \h$.
Then we have:
\begin{enumerate}
\item 
$\phi_{|u|}^{t}: H^{s}(\R^{3},\R)\rightarrow D^{t,2}(\R^{3}, \R)$ is continuous and maps bounded sets into bounded sets,
\item if $u_{n}\rightharpoonup u$ in $\h$ then $\phi_{|u_{n}|}^{t}\rightharpoonup \phi_{|u|}^{t}$ in $D^{t,2}(\R^{3},\R)$,
\item $\phi^{t}_{|ru|}=r^{2}\phi^{t}_{|u|}$ for all $r\in \R$ and $\phi^{t}_{|u(\cdot+y)|}(x)=\phi^{t}_{|u|}(x+y)$,
\item $\phi_{|u|}^{t}\geq 0$ for all $u\in \h$, and we have
$$
\|\phi_{|u|}^{t}\|_{D^{t,2}}\leq C\|u\|_{L^{\frac{12}{3+2t}}(\R^{3})}^{2}\leq C\|u\|^{2}_{\e} \, \mbox{ and } \int_{\R^{3}} \phi_{|u|}^{t}|u|^{2}dx\leq C\|u\|^{4}_{L^{\frac{12}{3+2t}}(\R^{3})}\leq C\|u\|_{\e}^{4}.
$$
\end{enumerate} 
\end{lem}
\begin{proof}
$(1)$ Since $\phi_{|u|}^{t}\in D^{t,2}(\R^{3}, \R)$ satisfies \eqref{FPE}, that is
\begin{align}\label{spput}
\int_{\R^{3}}(-\Delta)^{\frac{t}{2}}\phi_{|u|}^{t} (-\Delta)^{\frac{t}{2}}v \,dx=\int_{\R^{3}} |u|^{2}v \,dx
\end{align}
for all $v\in D^{t,2}(\R^{3}, \R)$, we can see that $\mathcal{L}_{|u|}$ is such that $\|\mathcal{L}_{|u|}\|_{\mathcal{L}(D^{t,2}, \R)}=\|\phi_{|u|}^{t}\|_{D^{t,2}}$ for all $u\in \h$. Hence, in order to prove the continuity of $\phi_{|u|}^{t}$, it is enough to show that the map $u\in \h\mapsto \mathcal{L}_{|u|}\in \mathcal{L}(D^{t,2}, \R)$ is continuous. Let $u_{n}\rightarrow u$ in $\h$. Using Lemma \ref{DI} and Theorem \ref{Sembedding} we deduce that $|u_{n}|\rightarrow |u|$ in $L^{\frac{12}{3+2t}}(\R^{3})$. Hence, for all $v\in D^{t,2}(\R^{3}, \R)$ we have
\begin{align*}
|\mathcal{L}_{|u_{n}|}(v)-\mathcal{L}_{|u|}(v)|&=\left|\int_{\R^{3}} (|u_{n}|^{2}-|u|^{2})v\, dx\right| \\
&\leq \left(\int_{\R^{3}} ||u_{n}|^{2}-|u|^{2}|^{\frac{6}{3+2t}} \, dx\right)^{\frac{3+2t}{6}} \|v\|_{L^{\frac{6}{3-2t}}(\R^{3})} \\
&\leq C\left[\left(\int_{\R^{3}} ||u_{n}|-|u||^{\frac{12}{3+2t}} \, dx\right)^{\frac{1}{2}} \left(\int_{\R^{3}} ||u_{n}|+|u||^{\frac{12}{3+2t}} \, dx\right)^{\frac{1}{2}}\right]^{\frac{3+2t}{6}} \|v\|_{D^{t,2}} \\
&\leq C  \||u_{n}|-|u|\|_{L^{\frac{12}{3+2t}}(\R^{3})}\|v\|_{D^{t,2}}
\end{align*}
which implies that $\|\phi_{|u_{n}|}^{t}-\phi_{|u|}^{t} \|_{D^{t,2}}=\|\mathcal{L}_{|u_{n}|}-\mathcal{L}_{|u|}\|_{\mathcal{L}(D^{t,2}, \R)}\rightarrow 0$ as $n\rightarrow \infty$.\\
$(2)$ If $u_{n}\rightharpoonup u$ in $\h$, then Lemma \ref{DI} and Theorem \ref{Sembedding} yield $|u_{n}|\rightarrow |u|$ in $L^{q}_{loc}(\R^{3}, \R)$ for all $q\in [1, \2)$. Hence, for all $v\in C^{\infty}_{c}(\R^{3}, \R)$ we get
\begin{align*}
\langle \phi_{|u_{n}|}^{t}-\phi_{|u|}^{t}, v\rangle&=\int_{\R^{3}} (|u_{n}|^{2}-|u|^{2})v\, dx \\
&\leq \left(\int_{supp(v)} ||u_{n}|-|u||^{2}\, dx\right)^{\frac{1}{2}} \left(\int_{\R^{3}} ||u_{n}|+|u||^{2}\, dx\right)^{\frac{1}{2}}  \|v\|_{L^{\infty}(\R^{3})}\\
&\leq C \||u_{n}|-|u|\|_{L^{2}(supp(v))} \|v\|_{L^{\infty}(\R^{3})}\rightarrow 0.
\end{align*}
$(3)$ is obtained by the definition of $\phi_{|u|}^{t}$.\\
$(4)$ It is clear that $\phi_{|u|}^{t}\geq 0$. 
Using \eqref{spput} with $v=\phi_{|u|}^{t}$, H\"older inequality and \eqref{ruiz} we have
\begin{align*}
\|\phi_{|u|}^{t}\|^{2}_{D^{t,2}}&\leq \|u\|^{2}_{L^{\frac{12}{3+2t}}(\R^{3})} \|\phi_{|u|}^{t}\|_{L^{2^{*}_{t}}(\R^{3})}
\leq C  \|u\|^{2}_{L^{\frac{12}{3+2t}}(\R^{3})} \|\phi_{|u|}^{t}\|_{D^{t,2}}\leq C \|u\|^{2}_{\e}  \|\phi_{|u|}^{t}\|_{D^{t,2}}.
\end{align*}
On the other hand, in view of  \eqref{RF}, Hardy-Littlewood-Sobolev inequality \cite{LL} and \eqref{ruiz} we get
\begin{align*}
\int_{\R^{3}} \phi_{|u|}^{t}|u|^{2}dx&\leq C \||u|^{2}\|^{2}_{L^{\frac{6}{3+2t}}(\R^{3})}=C \|u\|^{4}_{L^{\frac{12}{3+2t}}(\R^{3})}\leq C\|u\|_{\e}^{4}. 
\end{align*} 
\end{proof}

\section{Variational framework for the modified functional}\label{sec3}

It is standard to check that weak solutions to (\ref{MPe}) can be found as critical points of the Euler-Lagrange functional
\begin{align*}
J_{\e}(u)=\frac{1}{2}\|u\|^{2}_{\e}+\frac{1}{4}\int_{\R^{3}}\phi_{|u|}^{t}|u|^{2}dx-\frac{1}{2}\int_{\R^{3}} G_{\e}(x, |u|^{2})\, dx,
\end{align*}
We also consider the autonomous problem associated to \eqref{MPe}, that is
\begin{equation}\label{AP0}
(-\Delta)^{s} u + V_{0}u +\phi^{t}_{|u|}u=  f(u^{2})u+|u|^{\2-2}u \mbox{ in } \R^{3}, 
\end{equation}
and we introduce the  corresponding energy functional $J_{V_{0}}: H^{s}(\R^{3}, \R)\rightarrow \R$ given by
\begin{align*}
J_{V_{0}}(u)&=\frac{1}{2}\int_{\R^{3}}|(-\Delta)^{\frac{s}{2}}u|^{2}+V_{0} |u|^{2}\, dx+\frac{1}{4}\int_{\R^{3}}\phi^{t}_{|u|}u^{2}dx-\frac{1}{2}\int_{\R^{3}} F(u^{2})\, dx-\frac{1}{\2}\int_{\R^{3}}  |u|^{\2}\, dx\\
&=\frac{1}{2}\|u\|^{2}_{V_{0}}+\frac{1}{4}\int_{\R^{3}}\phi^{t}_{|u|}u^{2}dx-\frac{1}{2}\int_{\R^{3}} F(u^{2})\, dx-\frac{1}{\2} \int_{\R^{3}} |u|^{\2}\, dx
\end{align*}
where we used the notation $\|\cdot\|_{V_{0}}$ to denote the $H^{s}(\R^{3}, \R)$-norm (equivalent to the standard one). 
We also denote by $J_{\mu}$ the functional associated to the problem \eqref{AP0} replacing $V_{0}$ by $\mu$.

Now, let us introduce the Nehari manifold associated to (\ref{Pe}), that is
\begin{equation*}
\mathcal{N}_{\e}:= \{u\in \h \setminus \{0\} : \langle J_{\e}'(u), u \rangle =0\},
\end{equation*}
and we denote by $\mathcal{N}_{V_{0}}$ the Nehari manifold associated to \eqref{AP0}.
Using the growth conditions of $g$, we can show that there exists $r>0$ 
independent of  $u$ such that 
\begin{equation}\label{uNr}
\|u\|_{\e}\geq r \mbox{ for all } u\in \mathcal{N}_{\e}.
\end{equation}
Indeed, fixed $u\in \mathcal{N}_{\e}$, we get
\begin{align*}
0&=\|u\|_{\e}^{2}+\int_{\R^{3}} \phi^{t}_{|u|}|u|^{2}\, dx-\int_{\R^{3}} g_{\e}(x, |u|^{2})|u|^{2}\,dx\\
&\geq \|u\|_{\e}^{2}-\frac{1}{k} \int_{\R^{3}} V_{\e}(x)|u|^{2}\, dx-C\|u\|_{L^{\2}(\R^{3})}^{\2} \\
&\geq \frac{k-1}{k}\|u\|^{2}_{\e}-C\|u\|_{\e}^{\2}.
\end{align*}
In what follows, we show that $J_{\e}$ possesses a mountain pass geometry \cite{AR}. 
\begin{lem}\label{MPG}
\begin{compactenum}[$(i)$]
\item $J_{\e}(0)=0$;
\item there exists $\alpha, \rho>0$ such that $J_{\e}(u)\geq \alpha$ for any $u\in \h$ such that $\|u\|_{\e}=\rho$;
\item there exists $e\in \h$ with $\|e\|_{\e}>\rho$ such that $J_{\e}(e)<0$.
\end{compactenum}
\end{lem}
\begin{proof}
Using $(g_1)$, $(g_2)$, and Theorem \ref{Sembedding} we can see that for any $\delta>0$ there exists $C_{\delta}>0$ such that
$$
J_{\e}(u)\geq \frac{1}{2}\|u\|^{2}_{\e}-\delta C\|u\|^{4}_{\e}-C_{\delta} \|u\|^{\2}_{\e}.
$$
Choosing $\delta>0$ sufficiently small, we can see that $(i)$ holds.
Regarding $(ii)$, we can note that in view of $(f_3)$ and Lemma \ref{poisson}, we have for any $u\in \h\setminus\{0\}$ with $supp(u)\subset \Lambda_{\e}$ and $T>1$
\begin{align*}
J_{\e}(Tu)&\leq \frac{T^{2}}{2} \|u\|^{2}_{\e}+\frac{T^{4}}{4}\int_{\R^{3}}\phi_{|u|}^{t}|u|^{2}dx-\frac{1}{2}\int_{\Lambda_{\e}} F(T^{2}|u|^{2})\, dx \\
&\leq \frac{T^{4}}{2} \left(\|u\|^{2}_{\e}+\int_{\R^{3}}\phi_{|u|}^{t}|u|^{2}dx\right)-CT^{\theta} \int_{\Lambda_{\e}} |u|^{\theta}\, dx+C
\end{align*}
which together with $\theta>4$, implies that $J_{\e}(Tu)\rightarrow -\infty$ as $T\rightarrow \infty$.
\end{proof}

In view of Lemma \ref{MPG}, we can define the minimax level 
\begin{align*}
c_{\e}=\inf_{\gamma\in \Gamma_{\e}} \max_{t\in [0, 1]} J_{\e}(\gamma(t)) \quad \mbox{ where } \quad \Gamma_{\e}=\{\gamma\in C([0, 1], \h): \gamma(0)=0 \mbox{ and } J_{\e}(\gamma(1))<0\}.
\end{align*}
It is standard to verify that $c_{\e}$ can be characterized as follows:
$$
c_{\e}=\inf_{u\in \h\setminus\{0\}} \sup_{t\geq 0} J_{\e}(t u)=\inf_{u\in \N_{\e}} J_{\e}(u);
$$
see \cite{W} for more details.
Using a version of the Mountain Pass Theorem without $(PS)$ condition (see \cite{W}), we can deduce the existence of a Palais-Smale sequence sequence $(u_{n})$ at the level $c_{\e}$.

Now, we show that $J_{\e}$ verifies a compactness condition which is related to the best constant $S_{*}$ of the Sobolev embedding $H^{s}(\R^{3}, \R)\subset L^{\2}(\R^{3}, \R)$ (see \cite{DPV}). More precisely:
\begin{lem}\label{PSc}
Let $c<c_{*}=\frac{s}{3}S_{*}^{\frac{3}{2s}}$. Then $J_{\e}$ satisfies the Palais-Smale condition at the level $c$.
\end{lem}
\begin{proof}
Let $(u_{n})\subset \h$ be a $(PS)_{c}$-sequence of $J_{\e}$, that is 
$$
J_{\e}(u_{n})\rightarrow c<\frac{s}{3}S_{*}^{\frac{3}{2s}} \mbox{ and } J_{\e}'(u_{n})\rightarrow 0.
$$
We divide the proof into three steps.\\
{\bf Step 1} The sequence $(u_{n})$ is bounded in $\h$. 
Indeed, using $(g_3)$ we can see that
\begin{align*}
c+o_{n}(1)\|u_{n}\|_{\e}&= J_{\e}(u_{n})-\frac{1}{\theta}\langle J'_{\e}(u_{n}), u_{n}\rangle \\
&= \left(\frac{1}{2}-\frac{1}{\theta}\right)\|u_{n}\|^{2}_{\e}+\left( \frac{1}{4}-\frac{1}{\theta}\right)\int_{\R^{3}} \phi_{|u_{n}|}^{t}|u_{n}|^{2}dx\\
&+\frac{1}{\theta}\int_{\R^{3}} \left[g_{\e}(x, |u_{n}|^{2})|u_{n}|^{2}-\frac{\theta}{2} G_{\e}(x, |u_{n}|^{2})\right]\, dx\\
&\geq \left(\frac{1}{2}-\frac{1}{\theta}\right)\|u_{n}\|^{2}_{\e}
+\left(\frac{2-\theta}{2\theta} \right)\int_{\Lambda^{c}_{\e}} G_{\e}(x, |u_{n}|^{2})\, dx \\
&\geq \left(\frac{1}{2}-\frac{1}{\theta}\right)\|u_{n}\|^{2}_{\e}+\left(\frac{2-\theta}{2\theta k} \right)\int_{\Lambda^{c}_{\e}} V_{\e}(x) |u_{n}|^{2}\, dx \\
&\geq \left(\frac{\theta-2}{2\theta}\right)\left( 1-\frac{1}{k} \right)\|u_{n}\|^{2}_{\e}.
\end{align*}
Then, recalling that $k>\frac{\theta}{\theta-2}>1$, we get the thesis.

{\bf Step 2} 
For any $\xi>0$ there exists $R=R_{\xi}>0$ such that $\Lambda_{\e}\subset B_{R}$ and
\begin{align}\label{T}
\limsup_{n\rightarrow \infty}\int_{B_{R}^{c}} \int_{\R^{3}} \frac{|u_{n}(x)-u_{n}(y)e^{\imath A_{\e}(\frac{x+y}{2})\cdot (x-y)}|^{2}}{|x-y|^{3+2s}} \, dx dy+\int_{B_{R}^{c}} V_{\e}(x)|u_{n}|^{2}\, dx\leq \xi.
\end{align}
Let $\eta_{R}\in C^{\infty}(\R^{3}, \R)$ such that $0\leq \eta_{R}\leq 1$, $\eta_{R}=0$ in $B_{\frac{R}{2}}$, $\eta_{R}=1$ in $B_{R}^{c}$ and $|\nabla \eta_{R}|\leq \frac{C}{R}$ for some $C>0$ independent of $R$.
Since $\langle J'_{\e}(u_{n}), \eta_{R}u_{n}\rangle =o_{n}(1)$ we have
\begin{align*}
&\Re\left(\iint_{\R^{6}} \frac{(u_{n}(x)-u_{n}(y)e^{\imath A_{\e}(\frac{x+y}{2})\cdot (x-y)})\overline{(u_{n}(x)\eta_{R}(x)-u_{n}(y)\eta_{R}(y)e^{\imath A_{\e}(\frac{x+y}{2})\cdot (x-y)})}}{|x-y|^{3+2s}}\, dx dy \right)\\
&+\int_{\R^{3}} \phi_{|u_{n}|}^{t} |u_{n}|^{2}\eta_{R} dx+\int_{\R^{3}} V_{\e}(x)\eta_{R} |u_{n}|^{2}\, dx=\int_{\R^{N}} g_{\e}(x, |u_{n}|^{2})|u_{n}|^{2}\eta_{R}\, dx+o_{n}(1).
\end{align*}
Let us note that
\begin{align*}
&\Re\left(\iint_{\R^{6}} \frac{(u_{n}(x)-u_{n}(y)e^{\imath A_{\e}(\frac{x+y}{2})\cdot (x-y)})\overline{(u_{n}(x)\eta_{R}(x)-u_{n}(y)\eta_{R}(y)e^{\imath A_{\e}(\frac{x+y}{2})\cdot (x-y)})}}{|x-y|^{3+2s}}\, dx dy \right)\\
&=\Re\left(\iint_{\R^{6}} \overline{u_{n}(y)}e^{-\imath A_{\e}(\frac{x+y}{2})\cdot (x-y)}\frac{(u_{n}(x)-u_{n}(y)e^{\imath A_{\e}(\frac{x+y}{2})\cdot (x-y)})(\eta_{R}(x)-\eta_{R}(y))}{|x-y|^{3+2s}}  \,dx dy\right)\\
&+\iint_{\R^{6}} \eta_{R}(x)\frac{|u_{n}(x)-u_{n}(y)e^{\imath A_{\e}(\frac{x+y}{2})\cdot (x-y)}|^{2}}{|x-y|^{3+2s}}\, dx dy,
\end{align*}
so, using $(g_3)$-(ii) and Lemma \ref{poisson} we obtain
\begin{align}\label{PS1}
&\iint_{\R^{6}} \eta_{R}(x)\frac{|u_{n}(x)-u_{n}(y)e^{\imath A_{\e}(\frac{x+y}{2})\cdot (x-y)}|^{2}}{|x-y|^{3+2s}}\, dx dy+\int_{\R^{3}} V_{\e}(x)\eta_{R} |u_{n}|^{2}\, dx\nonumber\\
&\leq -\Re\left(\iint_{\R^{6}} \overline{u_{n}(y)}e^{-\imath A_{\e}(\frac{x+y}{2})\cdot (x-y)}\frac{(u_{n}(x)-u_{n}(y)e^{\imath A_{\e}(\frac{x+y}{2})\cdot (x-y)})(\eta_{R}(x)-\eta_{R}(y))}{|x-y|^{3+2s}}  \,dx dy\right) \nonumber\\
&+\frac{1}{k}\int_{\R^{3}} V_{\e}(x) \eta_{R} |u_{n}|^{2}\, dx+o_{n}(1).
\end{align}
From the H\"older inequality and the boundedness of $(u_{n})$ in $\h$ it follows that
\begin{align}\label{PS2}
&\left|\Re\left(\iint_{\R^{6}} \overline{u_{n}(y)}e^{-\imath A_{\e}(\frac{x+y}{2})\cdot (x-y)}\frac{(u_{n}(x)-u_{n}(y)e^{\imath A_{\e}(\frac{x+y}{2})\cdot (x-y)})(\eta_{R}(x)-\eta_{R}(y))}{|x-y|^{3+2s}}  \,dx dy\right)\right| \nonumber\\
&\leq \left(\iint_{\R^{6}} \frac{|u_{n}(x)-u_{n}(y)e^{\imath A_{\e}(\frac{x+y}{2})\cdot (x-y)}|^{2}}{|x-y|^{3+2s}}\,dxdy  \right)^{\frac{1}{2}} \left(\iint_{\R^{6}} |\overline{u_{n}(y)}|^{2}\frac{|\eta_{R}(x)-\eta_{R}(y)|^{2}}{|x-y|^{3+2s}} \, dxdy\right)^{\frac{1}{2}} \nonumber\\
&\leq C \left(\iint_{\R^{6}} |u_{n}(y)|^{2}\frac{|\eta_{R}(x)-\eta_{R}(y)|^{2}}{|x-y|^{3+2s}} \, dxdy\right)^{\frac{1}{2}}.
\end{align}
Arguing as in Lemma $4.3$ in \cite{A-CPAA} (see formula $(42)$ there) or Lemma $2.1$ in \cite{A3}, we can prove that
\begin{equation}\label{PS3}
\limsup_{R\rightarrow \infty}\limsup_{n\rightarrow \infty} \iint_{\R^{6}} |u_{n}(y)|^{2}\frac{|\eta_{R}(x)-\eta_{R}(y)|^{2}}{|x-y|^{3+2s}} \, dxdy=0.
\end{equation}
Then, in view of \eqref{PS1}, \eqref{PS2} and \eqref{PS3} we can conclude that
$$
\limsup_{R\rightarrow \infty}\limsup_{n\rightarrow \infty} \left(1-\frac{1}{k}\right) 
\int_{B_{R}^{c}} \int_{\R^{3}} \frac{|u_{n}(x)-u_{n}(y)e^{\imath A_{\e}(\frac{x+y}{2})\cdot (x-y)}|^{2}}{|x-y|^{3+2s}} \, dx dy+\int_{B_{R}^{c}} V_{\e}(x)|u_{n}|^{2}\, dx=0
$$
that is \eqref{T} is satisfied.

{\bf Step 3}: Up to subsequence, $u_{n}$ strongly converges in $\h$.\\
Using $u_{n}\rightharpoonup u$ in $\h$, Theorem \ref{Sembedding} and $(g_1)$-$(g_2)$, it is easy to see that
\begin{align}\label{Poissw0}
(u_{n}, \psi)_{\e}\rightarrow (u, \psi)_{\e} \mbox{ and } \Re\left(\int_{\R^{3}} g_{\e}(x, |u_{n}|^{2}) u_{n}\bar{\psi} dx\right)\rightarrow  \Re\left(\int_{\R^{3}} g_{\e}(x, |u|^{2}) u\bar{\psi} dx\right). 
\end{align}
Moreover, using \eqref{T} and Theorem \ref{Sembedding} we can see that for all $\xi>0$ there exists $R=R_{\xi}>0$ such that for any $n$ large enough
\begin{align*}
\|u_{n}-u\|_{L^{q}(\R^{3})}&=\|u_{n}-u\|_{L^{q}(B_{R})}+\|u_{n}-u\|_{L^{q}(B_{R}^{c})}\\
&\leq \|u_{n}-u\|_{L^{q}(B_{R})}+(\|u_{n}\|_{L^{q}(B_{R}^{c})}+\|u\|_{L^{q}(B^{c}_{R})}) \\
&\leq \xi+2C\xi,
\end{align*}
where $q\in [2, \2)$, which gives  
\begin{equation}\label{CSSq}
u_{n}\rightarrow u \mbox{ in } L^{q}(\R^{3}, \C)\quad \forall q\in [2, \2).
\end{equation}
Since $||u_{n}|-|u||\leq |u_{n}-u|$ and $\frac{12}{3+2t}\in (2, 2^{*}_{s})$, we also have $|u_{n}|\rightarrow |u|$ in $L^{\frac{12}{3+2t}}(\R^{3},\R)$. \\
Then, recalling that $\phi_{|u|}: L^{\frac{12}{3+2t}}(\R^{3}, \R)\rightarrow D^{t,2}(\R^{3}, \R)$ is continuous (see Lemma \ref{poisson}) we can deduce that
\begin{align}\begin{split}\label{Poissw1}
\phi_{|u_{n}|}^{t}\rightarrow \phi_{|u|}^{t} \mbox{ in } D^{t,2}(\R^{3}, \R).
\end{split}\end{align}
Putting together \eqref{CSSq}, \eqref{Poissw1}, H\"older inequality and Theorem \ref{Sembedding} we obtain
\begin{align}\label{Poissw2}
\Re\left(\int_{\R^{3}} (\phi_{|u_{n}|}^{t}u_{n}-\phi_{|u|}^{t}u)\bar{\psi}dx  \right)&=\Re\left(\int_{\R^{3}} \phi_{|u_{n}|}^{t}(u_{n}-u)\bar{\psi}+\int_{\R^{3}} (\phi_{|u_{n}|}^{t}-\phi_{|u|}^{t})u\bar{\psi}dx  \right) \nonumber\\
&\leq \|\phi_{|u_{n}|}^{t}\|_{L^{\frac{6}{3+2t}}(\R^{3})} \|u_{n}-u\|_{L^{\frac{12}{3+2t}}(\R^{3})}\|\psi\|_{L^{\frac{12}{3+2t}}(\R^{3})}\nonumber \\
&+\|\phi_{|u_{n}|}^{t}-\phi_{|u|}^{t}\|_{\frac{6}{3+2t}} \|u\|_{L^{\frac{12}{3+2t}}(\R^{3})}\|\psi\|_{\frac{12}{3+2t}} \nonumber\\
&\leq C \|u_{n}-u\|_{L^{\frac{12}{3+2t}}(\R^{3})}+C\|\phi_{|u_{n}|}^{t}-\phi_{|u|}^{t}\|_{D^{t,2}}\rightarrow 0.
\end{align}
Therefore, using $\langle J'_{\e}(u_{n}), \psi\rangle =o_{n}(1)$ for all $\psi\in C^{\infty}_{c}(\R^{3}, \C)$, and taking into account \eqref{Poissw0}, \eqref{Poissw1} and \eqref{Poissw2}, we can check that $J'_{\e}(u)=0$. In particular
\begin{equation}\label{Poiss0}
\|u\|^{2}_{\e}+\int_{\R^{3}} \phi_{|u|}^{t}|u|^{2}dx=\int_{\R^{3}} g_{\e}(x, |u|^{2})|u|^{2}\, dx.
\end{equation}
On the other hand, we know that $\langle J'_{\e}(u_{n}), u_{n}\rangle =o_{n}(1)$ implies that
\begin{equation}\label{Poiss1}
\|u_{n}\|^{2}_{\e}+\int_{\R^{3}} \phi_{|u_{n}|}^{t}|u_{n}|^{2}dx=\int_{\R^{3}} g_{\e}(x, |u_{n}|^{2})|u_{n}|^{2}\, dx+o_{n}(1),
\end{equation}
Now, we show that
\begin{equation}\label{Poiss3}
\int_{\R^{3}} \phi_{|u_{n}|}^{t}|u_{n}|^{2}dx\rightarrow \int_{\R^{3}} \phi_{|u|}^{t}|u|^{2}dx.
\end{equation}
Let us begin by proving that
\begin{equation*}
|\mathbb{D}(u_{n})-\mathbb{D}(u)|\leq \sqrt{\mathbb{D}(||u_{n}|^{2}-|u|^{2}|^{1/2})} \sqrt{\mathbb{D}(||u_{n}|^{2}+|u|^{2}|^{1/2})},
\end{equation*}
where
$$
\mathbb{D}(u)=\iint_{\R^{6}} |x-y|^{-(3-2t)}|u(x)|^{2}|u(y)|^{2}dxdy.
$$
Indeed, taking into account that $|x|^{-(3-2t)}$ is even and Theorem $9.8$ in \cite{LL} (see the remark after Theorem $9.8$ and recall that $-3<-(3-2t)<0$ ) we have
\begin{align*}
|\mathbb{D}(u_{n})-\mathbb{D}(u)|&=\left|\iint_{\R^{6}} |x-y|^{-(3-2t)}|u_{n}(x)|^{2}|u_{n}(y)|^{2}dxdy-\iint_{\R^{6}} |x-y|^{-(3-2t)}|u(x)|^{2}|u(y)|^{2}dxdy\right| \\
&=\Bigl|\iint_{\R^{6}} |x-y|^{-(3-2t)}|u_{n}(x)|^{2}|u_{n}(y)|^{2}dxdy+\iint_{\R^{6}} |x-y|^{-(3-2t)}|u_{n}(x)|^{2}|u(y)|^{2}dxdy \\
&-\iint_{\R^{6}} |x-y|^{-(3-2t)}|u(x)|^{2}|u_{n}(y)|^{2}dxdy-\iint_{\R^{6}} |x-y|^{-(3-2t)}|u(x)|^{2}|u(y)|^{2}dxdy\Bigr| \\
&=\left|\iint_{\R^{6}} |x-y|^{-(3-2t)}(|u_{n}(x)|^{2}-|u(x)|^{2}|)(|u_{n}(y)|^{2}+|u(y)|^{2}) dxdy  \right| \\
&\leq \iint_{\R^{6}} |x-y|^{-(3-2t)}||u_{n}(x)|^{2}-|u(x)|^{2}|| ||u_{n}(y)|^{2}+|u(y)|^{2}| dxdy  \\
&\leq C\sqrt{\mathbb{D}(||u_{n}|^{2}-|u|^{2}|^{1/2})} \sqrt{\mathbb{D}(||u_{n}|^{2}+|u|^{2}|^{1/2})}.
\end{align*}
Thus, using Hardy-Littlewood-Sobolev inequality (see Theorem $4.3$ in \cite{LL}), H\"older inequality, the boundedness of $(|u_{n}|)$ in $H^{s}(\R^{3}, \R)$ and $|u_{n}|\rightarrow |u|$ in $L^{\frac{12}{3+2t}}(\R^{3},\R)$ we can see that
\begin{align*}
|\mathbb{D}(u_{n})-\mathbb{D}(u)|^{2}&\leq C\|||u_{n}|^{2}-|u|^{2}||^{1/2}\|^{4}_{L^{\frac{12}{3+2t}}(\R^{3})} \|||u_{n}|^{2}+|u|^{2}||^{1/2}\|^{4}_{L^{\frac{12}{3+2t}}(\R^{3})} \nonumber \\
&\leq C\||u_{n}|-|u|\|^{2}_{L^{\frac{12}{3+2t}}(\R^{3})}\rightarrow 0. 
\end{align*}
Finally we show that 
\begin{equation}\label{Poiss2}
\lim_{n\rightarrow \infty} \int_{\R^{3}} g_{\e}(x, |u_{n}|^{2})|u_{n}|^{2}\, dx=\int_{\R^{3}} g_{\e}(x, |u|^{2})|u|^{2}\, dx.
\end{equation}
Using $(f_1)$, $(f_2)$, $(g_2)$ and Theorem \ref{Sembedding} we get 
\begin{equation}\label{3.6-1}
\int_{\R^{3} \setminus B_{R}} g_{\e}(x, |u_{n}|^{2})|u_{n}|^{2} \, dx \leq C(\delta+\delta^{\frac{q}{2}}+\delta^{\frac{2^{*}_{s}}{2}}), 
\end{equation}
for any $n$ big enough.
On the other hand, choosing $R$ large enough, we may assume that
\begin{equation}\label{3.6-2}
\int_{\R^{3} \setminus B_{R}} g_{\e}(x, |u|^{2})|u|^{2} \, dx \leq \delta. 
\end{equation}
From the arbitrariness of $\delta>0$, we can see that \eqref{3.6-1} and \eqref{3.6-2} yield
\begin{align}\label{ioo} 
\int_{\R^{3} \setminus B_{R}} g_{\e}(x, |u_{n}|^{2})|u_{n}|^{2} \, dx\rightarrow \int_{\R^{3} \setminus B_{R}} g_{\e}(x, |u|^{2})|u|^{2}\, dx
\end{align}
as $n\rightarrow \infty$.
Now, we note that from the definition of $g$ we know that
$$
g_{\e}(x, |u_{n}|^{2})|u_{n}|^{2}\leq f(|u_{n}|^{2})|u_{n}|^{2}+|u_{n}|^{2^{*}_{s}}+\frac{V_{0}}{K}|u_{n}|^{2} \mbox{ in } \R^{3}\setminus \Lambda_{\e}.
$$
Since $B_{R}\cap (\R^{3}\setminus \Lambda_{\e})$ is bounded, we can use $(f_1)$, $(f_2)$, $(g_2)$, the Dominated Convergence Theorem and the strong convergence in $L^{q}_{loc}(\R^{3}, \R)$ to see that 
\begin{align}\label{iooo}
\int_{B_{R}\cap (\R^{3}\setminus \Lambda_{\e})} g_{\e}(x, |u_{n}|^{2})|u_{n}|^{2}\, dx\rightarrow \int_{B_{R}\cap (\R^{3}\setminus \Lambda_{\e})} g_{\e}(x, |u|^{2})|u|^{2} \, dx
\end{align}
as $n\rightarrow \infty$.\\
At this point, we show that
\begin{equation}\label{ioooo}
\lim_{n\rightarrow \infty}\int_{\Lambda_{\e}} |u_{n}|^{2^{*}_{s}}\,dx=\int_{\Lambda_{\e}} |u|^{2^{*}_{s}} \,dx.
\end{equation}
Indeed, if we assume that \eqref{ioooo} is true, from Theorem \ref{Sembedding}, $(g_2)$, $(f_1)$, $(f_2)$ and the Dominated Convergence Theorem, we can see that
\begin{equation}\label{i5}
\int_{B_{R}\cap\Lambda_{\e}} g_{\e}(x, |u_{n}|^{2})|u_{n}|^{2} \, dx\rightarrow \int_{B_{R}\cap\Lambda_{\e}} g_{\e}(x, |u|^{2})|u|^{2} \, dx.
\end{equation}
Putting together \eqref{ioo}, \eqref{iooo} and \eqref{i5}, we can conclude that \eqref{Poiss2} holds.
Taking into account \eqref{Poiss0},  \eqref{Poiss1},  \eqref{Poiss3} and \eqref{Poiss2} we can deduce that
$$
\lim_{n\rightarrow \infty}\|u_{n}\|^{2}_{\e}=\|u\|^{2}_{\e}.
$$
In what follows we prove that \eqref{ioooo} is satisfied. From \eqref{T} and Lemma \ref{DI} we can see that $(|u_{n}|)$ is tight in $H^{s}(\R^{3}, \R)$, so 
by Concentration-Compactness Lemma \cite{A-CPAA, DMV, PP}, we can find an at most countable index set $I$, sequences $(x_{i})\subset \R^{3}$, $(\mu_{i}), (\nu_{i})\subset (0, \infty)$ such that 
\begin{align}\label{CML}
&\mu\geq |(-\Delta)^{\frac{s}{2}}|u||^{2}+\sum_{i\in I} \mu_{i} \delta_{x_{i}}, \nonumber \\
&\nu=|u|^{\2}+\sum_{i\in I} \nu_{i} \delta_{x_{i}} \quad \mbox{ and } S_{*} \nu_{i}^{\frac{2}{2^{*}_{s}}}\leq \mu_{i}
\end{align}
for any $i\in I$, where $\delta_{x_{i}}$ is the Dirac mass at the point $x_{i}$.
Let us show that $(x_{i})_{i\in I}\cap \Lambda_{\e}=\emptyset$. Assume by contradiction that 
$x_{i}\in \Lambda_{\e}$ for some $i\in I$. For any $\rho>0$, we define $\psi_{\rho}(x)=\psi(\frac{x-x_{i}}{\rho})$ where $\psi\in C^{\infty}_{0}(\R^{N}, [0, 1])$ is such that $\psi=1$ in $B_{1}$, $\psi=0$ in $\R^{3}\setminus B_{2}$ and $\|\nabla \psi\|_{L^{\infty}(\R^{3})}\leq 2$. We suppose that  $\rho>0$ is such that $supp(\psi_{\rho})\subset \Lambda_{\e}$. Since $(\psi_{\rho} u_{n})$ is bounded in $\h$, we can see that $\langle J'_{\e}(u_{n}),\psi_{\rho} u_{n}\rangle=o_{n}(1)$, so, using the pointwise diamagnetic inequality \cite{DS}, we get
\begin{align}\label{amicicritico}
\iint_{\R^{6}} &\psi_{\rho}(y)\frac{||u_{n}(x)|-|u_{n}(y)||^{2}}{|x-y|^{3+2s}} \, dx dy \nonumber\\
&\leq -\Re\left(\iint_{\R^{6}} \frac{(\psi_{\rho}(x)-\psi_{\rho}(y))(u_{n}(x)-u_{n}(y)e^{\imath A_{\e}(\frac{x+y}{2})\cdot (x-y)})}{|x-y|^{3+2s}} \overline{u_{n}(y)}e^{-\imath A_{\e}(\frac{x+y}{2})\cdot (x-y)} dx dy\right) \nonumber\\
&\quad +\int_{\R^{3}} \psi_{\rho} f(|u_{n}|^{2})|u_{n}|^{2}\, dx+\int_{\R^{3}} \psi_{\rho} |u_{n}|^{\2}\, dx+o_{n}(1).
\end{align}
Due to the fact that $f$ has subcritical growth and $\psi_{\rho}$ has compact support, we can see that
\begin{align}\label{Alessia1critico}
\lim_{\rho\rightarrow 0} &\lim_{n\rightarrow \infty} \int_{\R^{3}}  \psi_{\rho} f(|u_{n}|^{2})|u_{n}|^{2}\, dx=\lim_{\rho\rightarrow 0} \int_{\R^{3}}  \psi_{\rho} f(|u|^{2})|u|^{2}\, dx=0.
\end{align}
Now, we show that
\begin{equation}\label{niocritico}
\lim_{\rho\rightarrow 0}\lim_{n\rightarrow \infty} \Re\left(\iint_{\R^{6}} \frac{(\psi_{\rho}(x)-\psi_{\rho}(y))(u_{n}(x)-u_{n}(y)e^{\imath A_{\e}(\frac{x+y}{2})\cdot (x-y)})}{|x-y|^{3+2s}} \overline{u_{n}(y)}e^{-\imath A_{\e}(\frac{x+y}{2})\cdot (x-y)} dx dy\right)=0.
\end{equation}
Using H\"older inequality and the fact that $(u_{n})$ is bounded in $\h$, we can see that
\begin{align*}
&\left|\Re\left(\iint_{\R^{6}} \frac{(\psi_{\rho}(x)-\psi_{\rho}(y))(u_{n}(x)-u_{n}(y)e^{\imath A_{\e}(\frac{x+y}{2})\cdot (x-y)})}{|x-y|^{3+2s}} \overline{u_{n}(y)}e^{-\imath A_{\e}(\frac{x+y}{2})\cdot (x-y)} dx dy\right) \right|  \\
&\leq C\left(\iint_{\R^{6}} |u_{n}(y)|^{2} \frac{|\psi_{\rho}(x)-\psi_{\rho}(y)|^{2}}{|x-y|^{3+2s}} \, dx dy \right)^{\frac{1}{2}}.
\end{align*}
Arguing as in Lemma $4.3$ in \cite{A-CPAA} (see formula $(53)$ there) we can deduce that
\begin{equation}\label{NIOOcritico}
\lim_{\rho\rightarrow 0}\lim_{n\rightarrow \infty}  \iint_{\R^{6}} |u_{n}(x)|^{2} \frac{|\psi_{\rho}(x)-\psi_{\rho}(y)|^{2}}{|x-y|^{3+2s}} \, dx dy =0
\end{equation}
which implies that  \eqref{niocritico} holds.
Therefore, from \eqref{CML} and taking the limit as $n\rightarrow \infty$ and $\rho\rightarrow 0$ in \eqref{amicicritico} we can deduce that \eqref{Alessia1critico} and \eqref{niocritico} yield $\nu_{i}\geq \mu_{i}$ for all $i\in I$. In view of the last statement in \eqref{CML}, we have $\nu_{i}\geq S^{\frac{3}{2s}}$, and using Lemma \ref{DI} and $(g_3)$ we can deduce that
\begin{align*}
c&=J_{\e}(u_{n})-\frac{1}{4}\langle J'_{\e}(u_{n}), u_{n}\rangle+o_{n}(1) \\
&\geq \frac{1}{4}\|u_{n}\|^{2}_{\e}+\frac{1}{2}\int_{\R^{3}\setminus \Lambda_{\e}} \left[\frac{1}{2} g_{\e}(x, |u_{n}|^{2})|u_{n}|^{2}-G_{\e}(x, |u_{n}|^{2})\right] \, dx+ \frac{4s-3}{12}\int_{\Lambda_{\e}} |u_{n}|^{2^{*}_{s}} \, dx+o_{n}(1)\\
&\geq \left[\frac{1}{4}\int_{\Lambda_{\e}}\psi_{\rho} |(-\Delta)^{\frac{s}{2}} |u_{n}||^{2}dx+\frac{1}{4} \int_{\R^{3}\setminus \Lambda_{\e}} V_{\e}(x)|u_{n}|^{2} dx\right]-\frac{1}{4}\int_{\R^{3}\setminus \Lambda_{\e}} G_{\e}(x, |u_{n}|^{2}) dx \\
&+\frac{4s-3}{12}\int_{\Lambda_{\e}} |u_{n}|^{2^{*}_{s}}\, dx+o_{n}(1) \\
&\geq \frac{1}{4}\int_{\Lambda_{\e}}\psi_{\rho} |(-\Delta)^{\frac{s}{2}} |u_{n}||^{2}dx+\left(\frac{1}{4}-\frac{1}{4k}\right)\int_{\R^{3}\setminus \Lambda_{\e}} V_{\e}(x)|u_{n}|^{2} dx+\frac{4s-3}{12}\int_{\Lambda_{\e}} |u_{n}|^{2^{*}_{s}}\, dx+o_{n}(1) \\
&\geq \frac{1}{4}\int_{\Lambda_{\e}}\psi_{\rho} |(-\Delta)^{\frac{s}{2}} |u_{n}||^{2}dx+\frac{4s-3}{12} \int_{\Lambda_{\e}} \psi_{\rho} |u_{n}|^{2^{*}_{s}} \, dx+o_{n}(1).
\end{align*}
Then, in view of \eqref{CML}, $\nu_{i}\geq S^{\frac{3}{2s}}$ and taking the limit as $n\rightarrow \infty$, we find
\begin{align*}
c&\geq \frac{1}{4}\sum_{\{i\in I: x_{i}\in \Lambda_{\e}\}} \psi_{\rho}(x_{i}) \mu_{i}+ \frac{4s-3}{12}\sum_{\{i\in I: x_{i}\in \Lambda_{\e}\}} \psi_{\rho}(x_{i}) \nu_{i} \\
&\geq \frac{1}{4}\sum_{\{i\in I: x_{i}\in \Lambda_{\e}\}} \psi_{\rho}(x_{i}) S_{*}\nu_{i}^{2/2^{*}_{s}}+ \frac{4s-3}{12}\sum_{\{i\in I: x_{i}\in \Lambda_{\e}\}} \psi_{\rho}(x_{i}) \nu_{i} \\
&\geq \frac{1}{4} S_{*}^{\frac{3}{2s}}+ \frac{4s-3}{12} S_{*}^{\frac{3}{2s}}=\frac{1}{3} S_{*}^{\frac{3}{2s}},
\end{align*}
which gives a contradiction. This means that \eqref{ioooo} holds and we can conclude the proof.
\end{proof}

In view of Lemma \ref{MPG}, Lemma \ref{PSc} and that $c_{\e}<c_{*}$ for $\e>0$ small enough (see Lemma \ref{AMlem1} below), one can apply the Mountain Pass Theorem \cite{AR} to deduce the existence of a nontrivial solution to \eqref{MPe} for small $\e$. 
Nevertheless, to obtain multiple critical points, we need to work with the functional $J_{\e}$ constrained to $\N_{\e}$. Therefore, it is fundamental to prove the following compactness result:
\begin{prop}\label{propPSc}
Let $c\in \R$ be such that $c<c_{*}=\frac{s}{3}S_{*}^{\frac{3}{2s}}$. Then, the functional $J_{\e}$ restricted to $\mathcal{N}_{\e}$ satisfies the $(PS)_{c}$ condition at the level $c$.
\end{prop}
\begin{proof}
Let $(u_{n})\subset \mathcal{N}_{\e}$ be such that $J_{\e}(u_{n})\rightarrow c$ and $\|J'_{\e}(u_{n})_{|\mathcal{N}_{\e}}\|_{*}=o_{n}(1)$. Then there exists $(\lambda_{n})\subset \R$ such that
\begin{equation}\label{AFT}
J'_{\e}(u_{n})=\lambda_{n} T'_{\e}(u_{n})+o_{n}(1)
\end{equation}
where $T_{\e}: \h\rightarrow \R$ is given by
\begin{align*}
T_{\e}(u)=\|u\|_{\e}^{2}+\int_{\R^{3}} \phi_{|u|}^{t}|u|^{2}\, dx-\int_{\R^{3}} g_{\e}(x, |u|^{2})|u|^{2}\, dx.
\end{align*}
Then, using $\langle J'_{\e}(u_{n}), u_{n}\rangle=0$, the definition of $g$ and the monotonicity of $\eta$ we can see that
\begin{align*}
&\langle T'_{\e}(u_{n}), u_{n}\rangle \nonumber\\
&=2\|u_{n}\|_{\e}^{2}+4 \int_{\R^{3}} \phi_{|u_{n}|}^{t}|u_{n}|^{2}\, dx-2\int_{\R^{3}} g'_{\e}(x, |u_{n}|^{2})|u_{n}|^{4}\, dx-2\int_{\R^{3}} g_{\e}(x, |u_{n}|^{2})|u_{n}|^{2}\, dx  \label{22ZS} \\
&=-2\|u_{n}\|^{2}_{\e}+2\int_{\R^{3}} g_{\e}(x, |u_{n}|^{2})|u_{n}|^{2}\, dx-2\int_{\R^{3}} g'_{\e}(x, |u_{n}|^{2})|u_{n}|^{4}\, dx \nonumber \\
&= -2 \|u_{n}\|^{2}_{\e} + 2 \int_{\Lambda_{\e}\cup \{|u_{n}|^{2}<t_{a}\}} \left[g_{\e}(x, |u_{n}|^{2}) |u_{n}|^{2}- g'_{\e}(x, |u_{n}|^{2})|u_{n}|^{4} \right]\, dx \nonumber \\
&\quad +  2 \int_{\Lambda^{c}_{\e}\cap \{t_{a}\leq |u_{n}|^{2}\leq T_{a}\}} \left[g_{\e}(x, |u_{n}|^{2}) |u_{n}|^{2}- g'_{\e}(x, |u_{n}|^{2})|u_{n}|^{4} \right]\, dx \nonumber \\
&\quad+2 \int_{\Lambda^{c}_{\e}\cap \{|u_{n}|^{2}> T_{a}\}} \left[g_{\e}(x, |u_{n}|^{2}) |u_{n}|^{2}- g'_{\e}(x, |u_{n}|^{2})|u_{n}|^{4} \right]\, dx       \nonumber \\
&\leq -2 \|u_{n}\|^{2}_{\e}+ \frac{2}{k} \int_{\Lambda^{c}_{\e}\cap \{|u_{n}|^{2}> T_{a}\}}  V_{\e}(x) |u_{n}|^{2} \, dx \nonumber \\
&\quad +2 \int_{\Lambda_{\e}\cup \{|u_{n}|^{2}<t_{a}\}} \left[g_{\e}(x, |u_{n}|^{2}) |u_{n}|^{2}- g'_{\e}(x, |u_{n}|^{2})|u_{n}|^{4} \right]\, dx \\
&\quad +2 \int_{\Lambda^{c}_{\e}\cap \{t_{a}\leq |u_{n}|^{2}\leq T_{a}\}} \left[g_{\e}(x, |u_{n}|^{2}) |u_{n}|^{2}- g'_{\e}(x, |u_{n}|^{2})|u_{n}|^{4} \right]\, dx\leq  0 
\end{align*}
where we used $f'(t) t- f(t)\geq 0$ for any $t>0$ in view of $(f_4)$, condition $(\xi_{3})$, $f', \tilde{f}' \in C(\R^{3})$ and recalling the definition of $g$ we know that
\begin{equation*}
g'(x,t)=0 \quad \forall t\geq T_{a} \quad \mbox{ and } \quad g'(x,t)= \tilde{f}'(t) \quad \forall x\in \R^{3}\setminus \Lambda. 
\end{equation*}
Indeed, we obtain
\begin{align*}
\langle T'_{\e}(u_{n}), u_{n}\rangle&\leq \left(\frac{2}{k}-2 \right) \|u_{n}\|^{2}_{\e}+2\int_{\Lambda_{\e}\cup \{|u_{n}|^{2}<t_{a}\}} \left[f(|u_{n}|^{2}) |u_{n}|^{2}- f'(|u_{n}|^{2})|u_{n}|^{4} \right]\, dx \\
&\quad -\int_{\Lambda_{\e}\cup \{|u_{n}|^{2}<t_{a}\}} (2^{*}_{s}-4)|u_{n}|^{\2} \, dx \\
&\quad +2 \int_{\Lambda^{c}_{\e}\cap \{t_{a}\leq |u_{n}|^{2}\leq T_{a}\}} \left[\xi(|u_{n}|^{2}) |u_{n}|^{2}- \xi'(|u_{n}|^{2})|u_{n}|^{4} \right]\, dx \\
&\leq  \left(\frac{2}{k}-2 \right) \|u_{n}\|^{2}_{\e}-(2^{*}_{s}-4) \int_{\Lambda_{\e}\cup \{|u_{n}|^{2}<t_{a}\}} |u_{n}|^{\2} \, dx \\
&\leq -(2^{*}_{s}-4) \int_{\Lambda_{\e}} |u_{n}|^{\2} \, dx.
\end{align*}
Taking into account the above fact and the boundedness of $(u_{n})$ in $\h$, we can see that $\langle T'_{\e}(u_{n}), u_{n}\rangle\rightarrow \ell\leq 0$. If $\ell=0$ we can use $\langle J'_{\e}(u_{n}), u_{n}\rangle=0$ to deduce that 
\begin{equation*}
0\leq \left( 1- \frac{1}{k}\right) \|u_{n}\|^{2}_{\e} \leq o_{n}(1), 
\end{equation*}
that is $\|u_{n}\|_{\e}\rightarrow 0$, which is impossible due to \eqref{uNr}. As a consequence, $\ell<0$ and taking into account \eqref{AFT} we get $\lambda_{n}\rightarrow 0$, that is $u_{n}$ is a $(PS)_{c}$ sequence for the unconstrained functional.  The result follows from Lemma \ref{PSc}.
\end{proof}

As a consequence of the previous result we can see that
\begin{cor}\label{cor}
The critical points of the functional $J_{\e}$ on $\mathcal{N}_{\e}$ are critical points of $J_{\e}$.
\end{cor}

In what follows, we recall the following useful compactness result for the autonomous problem \eqref{AP0} whose proof can be obtained arguing as in Proposition 3.4 in \cite{LZ}.
\begin{lem}\label{FS}
Let $(u_{n})\subset \mathcal{N}_{\mu}$ be a sequence satisfying $J_{\mu}(u_{n})\rightarrow c < \frac{s}{3}S_{*}^{\frac{3}{2s}}$. Then, up to subsequences, the following alternatives holds:
\begin{compactenum}[(i)]
\item $(u_{n})$ strongly converges in $H^{s}(\R^{3}, \R)$, 
\item there exists a sequence $(\tilde{y}_{n})\subset \R^{3}$ such that,  up to a subsequence, $v_{n}(x)=u_{n}(x+\tilde{y}_{n})$ converges strongly in $H^{s}(\R^{3}, \R)$.
\end{compactenum}
In particular, there exists a minimizer $w\in H^{s}(\R^{3}, \R)$ for $J_{\mu}$ with $J_{\mu}(w)=c$.
\end{lem}

\noindent
Finally, we prove the following interesting relation between $c_{\e}$ and $c_{V_{0}}$.
\begin{lem}\label{AMlem1}
The numbers $c_{\e}$ and $c_{V_{0}}$ satisfy the following inequality
$$
\limsup_{\e\rightarrow 0} c_{\e}\leq c_{V_{0}}<c_{*}.
$$
\end{lem}
\begin{proof}
Firstly, we note that $c_{V_{0}}<\frac{s}{3}S_{*}^{\frac{3}{2s}}=c_{*}$ by Lemma $3.1$ in \cite{LZ}.
Now, in view of Lemma \ref{FS}, there exists a positive ground state $w\in H^{s}(\R^{3}, \R)$ to the autonomous problem \eqref{AP0}, so that $J'_{V_{0}}(w)=0$ and $J_{V_{0}}(w)=c_{V_{0}}$. 
Moreover, we know (see Proposition $3.4$ in \cite{LZ}) that $w\in C^{1, \gamma}(\R^{3}, \R)\cap L^{\infty}(\R^{3},\R)$, for some $\gamma>0$. Therefore, $|w(x)|\rightarrow 0$ as $|x|\rightarrow \infty$, 
and we can find $R>0$ such that $(-\Delta)^{s}w+\frac{V_{0}}{2}w\leq 0$ in $|x|>R$. Using Lemma 4.3 in \cite{FQT} we know that there exists a positive continuous function $\tilde{w}$ such that for $|x|>R$ (taking $R$ larger if it is necessary), it holds $(-\Delta)^{s}\tilde{w}+\frac{V_{0}}{2}\tilde{w}=0$ and $\tilde{w}(x)=\frac{C_{0}}{|x|^{3+2s}}$. 
In view of the continuity of $w$ and $\tilde{w}$ there exists some constant $C_{1}>0$ such that $z=w-C_{1}\tilde{w}\leq 0$ on $|x|=R$.
Moreover, we can see that $(-\Delta)^{s}z+\frac{V_{0}}{2}z\geq 0$ in $|x|\geq R$. Using the maximum principle we can deduce that $z\leq 0$ in $|x|\geq R$, that is 
\begin{equation}\label{remdecay}
0<w(x)\leq \frac{C}{|x|^{3+2s}} \quad \mbox{ for } |x|>>1.
\end{equation}
Let $\eta\in C^{\infty}_{c}(\R^{3}, [0,1])$ be a cut-off function such that $\eta=1$ in a neighborhood of zero $B_{\frac{\delta}{2}}$ and $\supp(\eta)\subset B_{\delta}\subset \Lambda$ for some $\delta>0$. 
Let us define $w_{\e}(x):=\eta_{\e}(x)w(x) e^{\imath A(0)\cdot x}$, with $\eta_{\e}(x)=\eta(\e x)$ for $\e>0$, and we observe that $|w_{\e}|=\eta_{\e}w$ and $w_{\e}\in \h$ in view of Lemma \ref{aux}. 
Now we prove that
\begin{equation}\label{limwr}
\lim_{\e\rightarrow 0}\|w_{\e}\|^{2}_{\e}=\|w\|_{V_{0}}^{2}\in(0, \infty).
\end{equation}
Since it is clear that $\int_{\R^{3}} V_{\e}(x)|w_{\e}|^{2}dx\rightarrow \int_{\R^{3}} V_{0} |w|^{2}dx$, we only need to show that
\begin{equation}\label{limwr*}
\lim_{\e\rightarrow 0}[w_{\e}]^{2}_{A_{\e}}=[w]^{2}.
\end{equation}
Using Lemma $5$ in \cite{PP} we know that 
\begin{equation}\label{PPlem}
[\eta_{\e} w]\rightarrow [w] \mbox{ as } \e\rightarrow 0.
\end{equation}
On the other hand
\begin{align*}
[w_{\e}]_{A_{\e}}^{2}
&=\iint_{\R^{6}} \frac{|e^{\imath A(0)\cdot x}\eta_{\e}(x)w(x)-e^{\imath A_{\e}(\frac{x+y}{2})\cdot (x-y)}e^{\imath A(0)\cdot y} \eta_{\e}(y)w(y)|^{2}}{|x-y|^{3+2s}} dx dy \nonumber \\
&=[\eta_{\e} w]^{2}
+\iint_{\R^{6}} \frac{\eta_{\e}^2(y)w^2(y) |e^{\imath [A_{\e}(\frac{x+y}{2})-A(0)]\cdot (x-y)}-1|^{2}}{|x-y|^{3+2s}} dx dy\\
&\quad+2\Re \iint_{\R^{6}} \frac{(\eta_{\e}(x)w(x)-\eta_{\e}(y)w(y))\eta_{\e}(y)w(y)(1-e^{-\imath [A_{\e}(\frac{x+y}{2})-A(0)]\cdot (x-y)})}{|x-y|^{3+2s}} dx dy \\
&=: [\eta_{\e} w]^{2}+X_{\e}+2Y_{\e}.
\end{align*}
Then, in view of 
$|Y_{\e}|\leq [\eta_{\e} w] \sqrt{X_{\e}}$ and \eqref{PPlem}, it is suffices to prove that $X_{\e}\rightarrow 0$ as $\e\rightarrow 0$ to deduce that \eqref{limwr*} holds.\\
Let us note that for $0<\beta<\alpha/({1+\alpha-s})$, 
\begin{equation}\label{Ye}
\begin{split}
X_{\e}
&\leq \int_{\R^{3}} w^{2}(y) dy \int_{|x-y|\geq\e^{-\beta}} \frac{|e^{\imath [A_{\e}(\frac{x+y}{2})-A(0)]\cdot (x-y)}-1|^{2}}{|x-y|^{3+2s}} dx\\
&+\int_{\R^{3}} w^{2}(y) dy  \int_{|x-y|<\e^{-\beta}} \frac{|e^{\imath [A_{\e}(\frac{x+y}{2})-A(0)]\cdot (x-y)}-1|^{2}}{|x-y|^{3+2s}} dx \\
&=:X^{1}_{\e}+X^{2}_{\e}.
\end{split}
\end{equation}
Using $|e^{\imath t}-1|^{2}\leq 4$ and $w\in H^{s}(\R^{3}, \R)$, we get
\begin{equation}\label{Ye1}
X_{\e}^{1}\leq C \int_{\R^{3}} w^{2}(y) dy \int_{\e^{-\beta}}^\infty \rho^{-1-2s} d\rho\leq C \e^{2\beta s} \rightarrow 0.
\end{equation}
Since $|e^{\imath t}-1|^{2}\leq t^{2}$ for all $t\in \R$, $A\in C^{0,\alpha}(\R^3,\R^3)$ for $\alpha\in(0,1]$, and $|x+y|^{2}\leq 2(|x-y|^{2}+4|y|^{2})$, we have
\begin{equation}\label{Ye2}
\begin{split}
X^{2}_{\e}&
	\leq \int_{\R^{3}} w^{2}(y) dy  \int_{|x-y|<\e^{-\beta}} \frac{|A_{\e}\left(\frac{x+y}{2}\right)-A(0)|^{2} }{|x-y|^{3+2s-2}} dx \\
	&\leq C\e^{2\alpha} \int_{\R^{3}} w^{2}(y) dy  \int_{|x-y|<\e^{-\beta}} \frac{|x+y|^{2\alpha} }{|x-y|^{3+2s-2}} dx \\
	&\leq C\e^{2\alpha} \left(\int_{\R^{3}} w^{2}(y) dy  \int_{|x-y|<\e^{-\beta}} \frac{1 }{|x-y|^{3+2s-2-2\alpha}} dx\right.\\
	&\qquad\qquad+ \left. \int_{\R^{3}} |y|^{2\alpha} w^{2}(y) dy  \int_{|x-y|<\e^{-\beta}} \frac{1}{|x-y|^{3+2s-2}} dx\right) \\
	&=: C\e^{2\alpha} (X^{2, 1}_{\e}+X^{2, 2}_{\e}).
	\end{split}
	\end{equation}	
	Then
	\begin{equation}\label{Ye21}
	X^{2, 1}_{\e}
	= C  \int_{\R^{3}} w^{2}(y) dy \int_0^{\e^{-\beta}} \rho^{1+2\alpha-2s} d\rho
	\leq C\e^{-2\beta(1+\alpha-s)}.
	\end{equation}
	On the other hand, using \eqref{remdecay}, we infer that
	\begin{equation}\label{Ye22}
	\begin{split}
	 X^{2, 2}_{\e}
	 &\leq C  \int_{\R^{3}} |y|^{2\alpha} w^{2}(y) dy \int_0^{\e^{-\beta}}\rho^{1-2s} d\rho  \\
	&\leq C \e^{-2\beta(1-s)} \left[\int_{B_1(0)}  w^{2}(y) dy + \int_{B_1^c(0)} \frac{1}{|y|^{2(3+2s)-2\alpha}} dy \right]  \\
	&\leq C \e^{-2\beta(1-s)}.
	\end{split}
	\end{equation}
	Taking into account \eqref{Ye}, \eqref{Ye1}, \eqref{Ye2}, \eqref{Ye21} and \eqref{Ye22} we can conclude that $X_{\e}\rightarrow 0$. Therefore \eqref{limwr} holds.  Moreover, by \eqref{PPlem}, the Dominated Convergence Theorem, and the fact that $\h$ is a Hilbert space, we can see that $|w_{\e}|=\eta_{\e}w$ strongly converges to $w$ in $H^{s}(\R^{3}, \R)$, so we deduce that
\begin{equation}\label{Poisslim}
\lim_{\e\rightarrow 0} \int_{\R^{3}}\phi_{|w_{\e}|}^{t} |w_{\e}|^{2}dx=\int_{\R^{3}}\phi_{w}^{t} w^{2}dx.
\end{equation}
Now, let $t_{\e}>0$ be the unique number such that 
\begin{equation*}
J_{\e}(t_{\e} w_{\e})=\max_{t\geq 0} J_{\e}(t w_{\e}).
\end{equation*}
Then $t_{\e}$ verifies 
\begin{equation}\label{AS1}
t_{\e}^{2}\|w_{\e}\|_{\e}^{2}+t_{\e}^{4}\int_{\R^{3}}\phi_{|w_{\e}|}^{t}|w_{\e}|^{2}dx=\int_{\R^{3}} g_{\e}(x, t_{\e}^{2} |w_{\e}|^{2}) |t_{\e}w_{\e}|^{2}dx=\int_{\R^{3}} f(t_{\e}^{2} |w_{\e}|^{2}) |t_{\e}w_{\e}|^{2}+|t_{\e}w_{\e}|^{\2} dx
\end{equation}
where we used $supp(\eta)\subset \Lambda$ and $g(x, t)=f(t)+t^{\frac{\2-2}{2}}$ on $\Lambda$.\\
Let us prove that $t_{\e}\rightarrow 1$ as $\e\rightarrow 0$. Using that $\eta=1$ in $B_{\frac{\delta}{2}}$, that $w$ is a continuous positive function, that $\frac{f(t^{2})}{t^{2}}\geq 0$ for $t>0$ and that $\2-4=\frac{2(4s-3)}{3-2s}>0$ we can see that
$$
\frac{1}{t_{\e}^{2}}\|w_{\e}\|_{\e}^{2}+\int_{\R^{3}}\phi_{|w_{\e}|}^{t}|w_{\e}|^{2}dx\geq t_{\e}^{\frac{2(4s-3)}{3-2s}} \alpha^{\2}_{0} |B_{\frac{\delta}{2}}|
$$
where $\alpha_{0}=\min_{\bar{B}_{\frac{\delta}{2}}} w>0$. So, if $t_{\e}\rightarrow \infty$ as $\e\rightarrow 0$ then we can use \eqref{limwr} and \eqref{Poisslim} to deduce that $\int_{\R^{3}}\phi_{w}^{t}w^{2}dx= \infty$ which gives a contradiction.
On the other hand, if $t_{\e}\rightarrow 0$ as $\e\rightarrow 0$ we can use \eqref{AS1}, the growth assumptions on $g$, \eqref{limwr}, \eqref{Poisslim} to infer that $\|w\|_{0}^{2}= 0$ which is impossible.
In conclusion $t_{\e}\rightarrow t_{0}\in (0, \infty)$ as $\e\rightarrow 0$.
Now, taking the limit as $\e\rightarrow 0$ in \eqref{AS1} and using \eqref{Poisslim}, \eqref{limwr}, we can see that 
\begin{equation*}
\frac{1}{t_{0}^{2}}\|w\|_{V_{0}}^{2}+\int_{\R^{3}}\phi_{w}^{t}w^{2}dx=\int_{\R^{3}} \frac{f(t_{0}^{2} w^{2})}{(t_{0}^{2}w^{2})} w^{4}dx+t_{0}^{\2-4}\int_{\R^{3}} |w_{0}|^{\2}dx.
\end{equation*}
By $w\in \mathcal{N}_{0}$ it follows that 
\begin{equation*}
\left(\frac{1}{t_{0}^{2}}-1\right)\|w\|_{V_{0}}^{2}+\int_{\R^{3}}\phi_{w}^{t}w^{2}dx=\int_{\R^{3}} \left(\frac{f(t_{0}^{2} w^{2})}{(t_{0}^{2}w^{2})}-\frac{f(w^{2})}{w^{2}}\right) w^{4}dx+(t_{0}^{\2-4}-1)\int_{\R^{3}} |w_{0}|^{\2}dx,
\end{equation*}
and in view of $(f_4)$, we can deduce that $t_{0}=1$. Then, applying the Dominated Convergence Theorem, we obtain that $\lim_{\e\rightarrow 0} J_{\e}(t_{\e} w_{\e})=J_{V_{0}}(w)=c_{V_{0}}$.
Since $c_{\e}\leq \max_{t\geq 0} J_{\e}(t w_{\e})=J_{\e}(t_{\e} w_{\e})$, we can conclude  that
$\limsup_{\e\rightarrow 0} c_{\e}\leq c_{V_{0}}$.
\end{proof}


\section{Multiple solutions for the modified problem}\label{sec4}
This section is devoted to apply the Ljusternik-Schnirelmann category theory to prove a multiplicity result for the problem \eqref{MPe}. We begin proving the following technical results.
\begin{lem}\label{prop3.3}
Let $\e_{n}\rightarrow 0$ and $(u_{n})\subset \mathcal{N}_{\e_{n}}$ be such that $J_{\e_{n}}(u_{n})\rightarrow c_{V_{0}}$. Then there exists $(\tilde{y}_{n})\subset \R^{3}$ such that $v_{n}(x)=|u_{n}|(x+\tilde{y}_{n})$ has a convergent subsequence in $H^{s}(\R^{3}, \R)$. Moreover, up to a subsequence, $y_{n}=\e_{n} \tilde{y}_{n}\rightarrow y_{0}$ for some $y_{0}\in M$.
\end{lem}
\begin{proof}
Taking into account that $\langle J'_{\e_{n}}(u_{n}), u_{n}\rangle=0$, that $J_{\e_{n}}(u_{n})= c_{V_{0}}+o_{n}(1)$, Lemma \ref{AMlem1} and arguing as in the first part of Lemma \ref{PSc}, it is easy to see that there exists $C>0$ (independent of $n$) such that $\|u_{n}\|_{\e_{n}}\leq C$ for all $n\in \mathbb{N}$. Moreover, from Lemma \ref{DI}, we also know that $(|u_{n}|)$ is bounded in $H^{s}(\R^{3}, \R)$.
Now, we prove that 
there exist a sequence $(\tilde{y}_{n})\subset \R^{3}$, and constants $R>0$ and $\gamma>0$ such that
\begin{equation}\label{sacchi}
\liminf_{n\rightarrow \infty}\int_{B_{R}(\tilde{y}_{n})} |u_{n}|^{2} \, dx\geq \gamma>0.
\end{equation}
If by contradiction \eqref{sacchi} does not hold, then for all $R>0$ we get
$$
\lim_{n\rightarrow \infty}\sup_{y\in \R^{3}}\int_{B_{R}(y)} |u_{n}|^{2} \, dx=0.
$$
From the boundedness $(|u_{n}|)$ and Lemma \ref{Lions} we can see that $|u_{n}|\rightarrow 0$ in $L^{q}(\R^{3}, \R)$ for any $q\in (2, 2^{*}_{s})$. 
This fact combined with $(f_1)$ and $(f_2)$ gives
\begin{align}\label{flimiti}
\lim_{n\rightarrow \infty}\int_{\R^{3}} f(|u_{n}|^{2})|u_{n}|^{2} \,dx=0= \lim_{n\rightarrow \infty}\int_{\R^{3}} F(|u_{n}|^{2}) \, dx.
\end{align}
Moreover $|u_{n}|\rightarrow 0$ in $L^{\frac{12}{3+2t}}(\R^{3}, \R)$, so using $(4)$-Lemma \ref{poisson} we deduce that
\begin{align}\label{CSSq1}
\int_{\R^{3}} \phi_{|u_{n}|}^{t}|u_{n}|^{2}dx\rightarrow 0.
\end{align}
Therefore
\begin{equation}\label{7ADOM}
\int_{\R^{3}} G_{\e_{n}}(x, |u_{n}|^{2})\, dx\leq \frac{1}{\2}\int_{\Lambda_{\e}\cup \{|u_{n}|^{2}\leq t_{a}\}} |u_{n}|^{\2}\, dx+\frac{V_{0}}{2k}\int_{\Lambda^{c}_{\e}\cap \{|u_{n}|^{2}> T_{a}\}} |u_{n}|^{2}\, dx+o_{n}(1)
\end{equation}
and
\begin{equation}\label{8ADOM}
\int_{\R^{3}} g_{\e_{n}}(x, |u_{n}|^{2})|u_{n}|^{2}\, dx=\int_{\Lambda_{\e}\cup \{|u_{n}|^{2}\leq t_{a}\}} |u_{n}|^{\2}\, dx+\frac{V_{0}}{k}\int_{\Lambda^{c}_{\e}\cap \{|u_{n}|^{2}> T_{a}\}} |u_{n}|^{2}\, dx+o_{n}(1).
\end{equation}
Using \eqref{CSSq1}, \eqref{8ADOM} and $\langle J'_{\e_{n}}(u_{n}), u_{n}\rangle=0$ we can deduce that
\begin{equation}\label{9ADOM}
\|u_{n}\|_{\e_{n}}^{2}-\frac{V_{0}}{k}\int_{\Lambda^{c}_{\e}\cap \{|u_{n}|^{2}> T_{a}\}} |u_{n}|^{2}\, dx=\int_{\Lambda_{\e}\cup \{|u_{n}|^{2}\leq t_{a}\}} |u_{n}|^{\2}\, dx.
\end{equation}
Let $\ell\geq 0$ be such that 
$$
\|u_{n}\|_{\e_{n}}^{2}-\frac{V_{0}}{k}\int_{\Lambda^{c}_{\e}\cap \{|u_{n}|^{2}> T_{a}\}} |u_{n}|^{2}\, dx\rightarrow \ell.
$$
If $\ell=0$, then $u_{n}\rightarrow 0$ in $\h$ so that $J_{\e_{n}}(u_{n})\rightarrow 0$ which contradicts $c_{V_{0}}>0$. Then $\ell>0$. In view of \eqref{9ADOM} we can see that $\int_{\Lambda_{\e}\cup \{|u_{n}|^{2}\leq t_{a}\}} |u_{n}|^{\2}\, dx\rightarrow \ell$.
Taking into account $J_{\e_{n}}(u_{n})\rightarrow c_{V_{0}}$, \eqref{7ADOM} and $\langle J'_{\e_{n}}(u_{n}), u_{n}\rangle=0$ we can deduce that $\ell\leq \frac{3}{s} c_{V_{0}}$. From Lemma \ref{DI} and the definition of $S_{*}$, we know that
$$
\|u_{n}\|_{\e_{n}}^{2}-\frac{V_{0}}{k}\int_{\Lambda^{c}_{\e}\cap \{|u_{n}|^{2}> T_{a}\}} |u_{n}|^{2}\, dx\geq S_{*} \left(\int_{\Lambda_{\e}\cup \{|u_{n}|^{2}\leq t_{a}\}} |u_{n}|^{\2}\, dx\right)^{2/\2},
$$
and letting the limit as $n\rightarrow \infty$ we find $\ell\geq S_{*}\ell^{2/\2}$ which combined with $\ell\leq \frac{3}{s} c_{V_{0}}$ implies that $c_{V_{0}}\geq \frac{s}{3}S_{*}^{\frac{3}{2s}}$ which is impossible in view of Lemma \ref{AMlem1}. Therefore \eqref{sacchi} holds.

Now, we set $v_{n}(x)=|u_{n}|(x+\tilde{y}_{n})$. Then $(v_{n})$ is bounded in $H^{s}(\R^{3}, \R)$, and we may assume that 
$v_{n}\rightharpoonup v\not\equiv 0$ in $H^{s}(\R^{3}, \R)$  as $n\rightarrow \infty$.
Fix $t_{n}>0$ such that $\tilde{v}_{n}=t_{n} v_{n}\in \mathcal{N}_{V_{0}}$. Using Lemma \ref{DI}, we can see that 
$$
c_{V_{0}}\leq J_{V_{0}}(\tilde{v}_{n})\leq \max_{t\geq 0}J_{\e_{n}}(tv_{n})= J_{\e_{n}}(u_{n})
$$
which together with Lemma \ref{AMlem1} implies that $J_{V_{0}}(\tilde{v}_{n})\rightarrow c_{V_{0}}$. In particular, $\tilde{v}_{n}\nrightarrow 0$ in $H^{s}(\R^{3}, \R)$.
Since $(v_{n})$ and $(\tilde{v}_{n})$ are bounded in $H^{s}(\R^{3}, \R)$ and $\tilde{v}_{n}\nrightarrow 0$  in $H^{s}(\R^{3}, \R)$, we deduce that $t_{n}\rightarrow t^{*}\geq 0$. Indeed $t^{*}>0$ since $\tilde{v}_{n}\nrightarrow 0$  in $H^{s}(\R^{3}, \R)$. From the uniqueness of the weak limit, we can deduce that $\tilde{v}_{n}\rightharpoonup \tilde{v}=t^{*}v\not\equiv 0$ in $H^{s}(\R^{3}, \R)$. 
This combined with Lemma \ref{FS} yields
\begin{equation}\label{elena}
\tilde{v}_{n}\rightarrow \tilde{v} \mbox{ in } H^{s}(\R^{3}, \R).
\end{equation} 
As a consequence, $v_{n}\rightarrow v$ in $H^{s}(\R^{3}, \R)$ as $n\rightarrow \infty$.

Now, we set $y_{n}=\e_{n}\tilde{y}_{n}$ and we show that $(y_{n})$ admits a subsequence, still denoted by $y_{n}$, such that $y_{n}\rightarrow y_{0}$ for some $y_{0}\in \Lambda$ such that $V(y_{0})=V_{0}$. Firstly, we prove that $(y_{n})$ is bounded. Assume by contradiction that, up to a subsequence, $|y_{n}|\rightarrow \infty$ as $n\rightarrow \infty$. Take $R>0$ such that $\Lambda \subset B_{R}(0)$. Since we may suppose that  $|y_{n}|>2R$, we have that for any $z\in B_{R/\e_{n}}$ 
$$
|\e_{n}z+y_{n}|\geq |y_{n}|-|\e_{n}z|>R.
$$
Now, using $(u_{n})\subset \N_{\e_{n}}$, $(V_{1})$, Lemma \ref{DI}, Lemma \ref{poisson}, the definition of $g$ and the change of variable $x\mapsto z+\tilde{y}_{n}$ we observe that 
\begin{align*}
[v_{n}]^{2}+\int_{\R^{3}} V_{0} v_{n}^{2}\, dx &\leq [v_{n}]^{2}+\int_{\R^{3}} V_{0} v_{n}^{2}\, dx+\int_{\R^{3}} \phi_{|v_{n}|}^{t}|v_{n}|^{2}dx \nonumber \\
&\leq\int_{\R^{3}} g(\e_{n} x+y_{n}, |v_{n}|^{2}) |v_{n}|^{2} \, dx \nonumber\\
&\leq \int_{B_{\frac{R}{\e_{n}}}(0)} \tilde{f}(|v_{n}|^{2}) |v_{n}|^{2} \, dx+\int_{\R^{3}\setminus B_{\frac{R}{\e_{n}}}(0)} f(|v_{n}|^{2}) |v_{n}|^{2}+|v_{n}|^{\2} \, dx \nonumber \\
&\leq \frac{V_{0}}{k}\int_{\R^{3}}|v_{n}|^{2}\, dx.
\end{align*}
which implies that
$v_{n}\rightarrow 0$ in $H^{s}(\R^{3}, \R)$, that is a contradiction. Therefore, $(y_{n})$ is bounded and we may assume that $y_{n}\rightarrow y_{0}\in \R^{3}$. If $y_{0}\notin \overline{\Lambda}$, then we can argue as before to infer that $v_{n}\rightarrow 0$ in $H^{s}(\R^{3}, \R)$, which is impossible. Hence $y_{0}\in \overline{\Lambda}$. 
Now, suppose by contradiction that $V(y_{0})>V_{0}$.
Then, using (\ref{elena}), Fatou's Lemma, the invariance of  $\R^{3}$ by translations, Lemma \ref{DI} and Lemma \ref{AMlem1}, we get 
\begin{align*}
c_{V_{0}}=J_{V_{0}}(\tilde{v})&<\frac{1}{2}[\tilde{v}]^{2}+\frac{1}{2}\int_{\R^{3}} V(y_{0})\tilde{v}^{2} \, dx+\frac{1}{4}\int_{\R^{3}}\phi_{|\tilde{v}|}^{t}\tilde{v}^{2}dx-\frac{1}{2}\int_{\R^{3}} F(|\tilde{v}|^{2})+\frac{1}{\2}\int_{\R^{3}} |\tilde{v}|^{\2}\, dx\\
&\leq \liminf_{n\rightarrow \infty}\Bigl[\frac{1}{2}[\tilde{v}_{n}]^{2}+\frac{1}{2}\int_{\R^{3}} V(\e_{n}x+y_{n}) |\tilde{v}_{n}|^{2} \, dx+\frac{1}{4}\int_{\R^{3}}\phi_{|\tilde{v}_{n}|}^{t}|\tilde{v}_{n}|^{2}dx\\
&-\frac{1}{2}\int_{\R^{3}} F(|\tilde{v}_{n}|^{2})+\frac{1}{\2}\int_{\R^{3}} |\tilde{v_{n}}|^{\2} \, dx  \Bigr] \\
&\leq \liminf_{n\rightarrow \infty}\Bigl[\frac{t_{n}^{2}}{2}[|u_{n}|]^{2}+\frac{t_{n}^{2}}{2}\int_{\R^{3}} V(\e_{n}z) |u_{n}|^{2} \, dz
+ \frac{t_{n}^{4}}{4}\int_{\R^{3}}\phi_{|u_{n}|}^{t}|u_{n}|^{2}dx \\
&-\frac{1}{2}\int_{\R^{3}} F(|t_{n} u_{n}|^{2})+\frac{t_{n}^{\2}}{\2}\int_{\R^{3}}|u_{n}|^{\2} \, dz  \Bigr] \\
&\leq \liminf_{n\rightarrow \infty} J_{\e_{n}}(t_{n} u_{n}) \leq \liminf_{n\rightarrow \infty} J_{\e_{n}}(u_{n})= c_{V_{0}}
\end{align*}
which gives a contradiction. Hence, $y_{0}\in M$ and this ends the proof of lemma.

\end{proof}

\noindent
Now, we aim to relate the number of positive solutions of \eqref{Pe} to the topology of the set $\Lambda$.
For this reason, we take $\delta>0$ such that
$$
M_{\delta}=\{x\in \R^{3}: {\rm dist}(x, M)\leq \delta\}\subset \Lambda,
$$
and we consider $\eta\in C^{\infty}_{0}(\R_{+}, [0, 1])$ such that $\eta(t)=1$ if $0\leq t\leq \frac{\delta}{2}$ and $\eta(t)=0$ if $t\geq \delta$.\\
For any $y\in \Lambda$, we introduce (see \cite{AD})
$$
\Psi_{\e, y}(x)=\eta(|\e x-y|) w\left(\frac{\e x-y}{\e}\right)e^{\imath \tau_{y} \left( \frac{\e x-y}{\e} \right)},
$$
where $\tau_{y}(x)=\sum_{j=1}^{3} A_{j}(x)x_{j}$ and $w\in H^{s}(\R^{3})$ is a positive ground state solution to the autonomous problem \eqref{AP0} (such a solution exists in view of Lemma \ref{FS}).

Let $t_{\e}>0$ be the unique number such that 
$$
\max_{t\geq 0} J_{\e}(t \Psi_{\e, y})=J_{\e}(t_{\e} \Psi_{\e, y}). 
$$
Finally, we consider $\Phi_{\e}: M\rightarrow \N_{\e}$ defined by setting
$$
\Phi_{\e}(y)= t_{\e} \Psi_{\e, y}.
$$
\begin{lem}\label{lem3.4}
The functional $\Phi_{\e}$ satisfies the following limit
\begin{equation*}
\lim_{\e\rightarrow 0} J_{\e}(\Phi_{\e}(y))=c_{V_{0}} \mbox{ uniformly in } y\in M.
\end{equation*}
\end{lem}
\begin{proof}
Assume by contradiction that there exist $\delta_{0}>0$, $(y_{n})\subset M$ and $\e_{n}\rightarrow 0$ such that 
\begin{equation}\label{puac}
|J_{\e_{n}}(\Phi_{\e_{n}}(y_{n}))-c_{V_{0}}|\geq \delta_{0}.
\end{equation}
Let us observe that by Lemma $4.1$ in \cite{AD} and the Dominated Convergence Theorem we get  
\begin{align}\begin{split}\label{nio3}
&\| \Psi_{\e_{n}, y_{n}} \|^{2}_{\e_{n}}\rightarrow \|w\|^{2}_{V_{0}}\in (0, \infty) \, \mbox{ and } \int_{\R^{3}}  \phi_{|\Psi_{\e_{n}, y_{n}}|}^{t} |\Psi_{\e_{n}, y_{n}}|^{2}dx\rightarrow \int_{\R^{3}}  \phi_{w}^{t} w^{2}dx \\
&\|\Psi_{\e_{n}, y_{n}}\|_{L^{2^{*}_{s}}(\R^{3})}\rightarrow \|w\|_{L^{2^{*}_{s}}(\R^{3})}. 
\end{split}\end{align}
Concerning the second limit in \eqref{nio3}, we note that $|\Psi_{\e, y}|=\eta(|\e x-y|) w\left(\frac{\e x-y}{\e}\right)$  converges strongly to  $w$ in $H^{s}(\R^{3}, \R)$, so we use the following property (see $(6)$ of Lemma $2.3$ in \cite{teng}):
$$
\mbox{ if } u_{n}\rightarrow u \mbox{ in } H^{s}(\R^{3}, \R)\mbox{ then } \int_{\R^{3}} \phi_{u_{n}}^{t}u_{n}^{2}\, dx\rightarrow  \int_{\R^{3}} \phi_{u}^{t}u^{2}\, dx.
$$ 
On the other hand, since $\langle J'_{\e_{n}}(\Phi_{\e_{n}}(y_{n})),\Phi_{\e_{n}}(y_{n})\rangle=0$ and using the change of variable $\displaystyle{z=\frac{\e_{n}x-y_{n}}{\e_{n}}}$ it follows that
\begin{align*}
&t_{\e_{n}}^{2}\|\Psi_{\e_{n}, y_{n}}\|_{\e_{n}}^{2} +t_{\e_{n}}^{4} \int_{\R^{3}} \phi^{t}_{|\Psi_{\e_{n}, y_{n}}|}|\Psi_{\e_{n}, y_{n}}|^{2}dz\nonumber \\
&=\int_{\R^{3}} g(\e_{n}z+y_{n}, |t_{\e_{n}}\eta(|\e_{n}z|)w(z)|^{2}) |t_{\e_{n}}\eta(|\e_{n}z|)w(z)|^{2} dz.
\end{align*}
If $z\in B_{\frac{\delta}{\e_{n}}}(0)\subset M_{\delta}\subset \Lambda$, then $\e_{n} z+y_{n}\in B_{\delta}(y_{n})\subset M_{\delta}\subset \Lambda_{\e}$. Thus, being $g(x,t)=f(t)+t^{\frac{\2-2}{2}}$ for all $x\in \Lambda$ and $\eta(t)=0$ for $t\geq \delta$, we get
\begin{align}\label{1nio}
&t_{\e_{n}}^{2}\|\Psi_{\e_{n}, y_{n}}\|_{\e_{n}}^{2} +t_{\e_{n}}^{4} \int_{\R^{3}} \phi^{t}_{|\Psi_{\e_{n}, y_{n}}|}|\Psi_{\e_{n}, y_{n}}|^{2}dz\nonumber \\
&=\int_{\R^{3}} f(|t_{\e_{n}}\eta(|\e_{n}z|)w(z)|^{2}) |t_{\e_{n}}\eta(|\e_{n}z|)w(z)|^{2}+|t_{\e_{n}}\eta(|\e_{n}z|)w(z)|^{\2} dz.
\end{align}
Since $\eta=1$ in $B_{\frac{\delta}{2}}(0)\subset B_{\frac{\delta}{\e_{n}}}(0)$ for all $n$ large enough, we get from \eqref{1nio}
\begin{align}\label{nioo}
&\frac{1}{t_{\e_{n}}^{2}} \|\Psi_{\e_{n}, y_{n}}\|_{\e}^{2}+\int_{\R^{3}} \phi_{|\Psi_{\e_{n}, y_{n}}|}^{t} \Psi_{\e_{n}, y_{n}}^{2}dx \nonumber \\
&=\int_{\R^{3}} \frac{f(|t_{\e_{n}}\Psi_{\e_{n},y_{n}}|^{2})+|t_{\e_{n}}\Psi_{\e_{n}, y_{n}}|^{\2-2}}{|t_{\e_{n}}\Psi_{\e_{n}, y_{n}}|^{2}}  |\Psi_{\e_{n}, y_{n}}|^{4}dx \nonumber \\
&\geq t_{\e_{n}}^{2^{*}_{s}-4} \int_{B_{\frac{\delta}{2}}(0)}  |w(z)|^{2^{*}_{s}} \, dz \nonumber \\
&\geq  t_{\e_{n}}^{\frac{2(4s-3)}{3-2s}} w(\hat{z})^{\2} |B_{\frac{\delta}{2}}(0)|, 
\end{align}
where
\begin{equation*}
w(\hat{z})=\min_{z\in B_{\frac{\delta}{2}}} w(z)>0.
\end{equation*} 
Now, assume by contradiction that $t_{\e_{n}}\rightarrow \infty$. 
So, using $t_{\e_{n}}\rightarrow \infty$, $s\in (\frac{3}{4},1)$, \eqref{nio3} and \eqref{nioo} we obtain 
$$
\int_{\R^{3}}  \phi_{w}^{t} w^{2}dx=\infty,
$$
that is a contradiction.
Therefore $(t_{\e_{n}})$ is bounded and, up to subsequence, we may assume that $t_{\e_{n}}\rightarrow t_{0}$ for some $t_{0}\geq 0$.  
Let us prove that $t_{0}>0$. Suppose by contradiction that $t_{0}=0$. 
Then, taking into account \eqref{nio3} and assumptions $(g_1)$ and $(g_2)$, we can see that \eqref{1nio} yields
\begin{align*}
\|t_{\e_{n}} \Psi_{\e_{n}, y_{n}}\|_{\e_{n}}^{2}\rightarrow 0
\end{align*}
which is impossible because of \eqref{uNr}. 
Hence $t_{0}>0$.
Thus, letting the limit as $n\rightarrow \infty$ in \eqref{1nio}, we deduce that
\begin{align*}
\frac{1}{t_{0}^{2}}\|w\|^{2}_{V_{0}}+\int_{\R^{3}} \phi_{w}^{t} w^{2}dx=\int_{\R^{3}} \frac{f((t_{0} w)^{2})+(t_{0}w)^{\2-2}}{(t_{0}w)^{2}} \,w^{4} \, dx.
\end{align*}
Taking into account that $w\in \N_{V_{0}}$ and condition $(f_4)$ we can infer that $t_{0}=1$.
Then, letting the limit as $n\rightarrow \infty$ and using that $t_{\e_{n}}\rightarrow 1$ 
we can conclude that
$$
\lim_{n\rightarrow \infty} J_{\e_{n}}(\Phi_{\e_{n}, y_{n}})=J_{V_{0}}(w)=c_{V_{0}},
$$
which contradicts \eqref{puac}.
\end{proof}

\noindent
At this point, we are in the position to define the barycenter map. For any $\delta>0$, we take $\rho=\rho(\delta)>0$ such that $M_{\delta}\subset B_{\rho}$, and we consider $\varUpsilon: \R^{3}\rightarrow \R^{3}$ defined by setting
\begin{equation*}
\varUpsilon(x)=
\left\{
\begin{array}{ll}
x &\mbox{ if } |x|<\rho \\
\frac{\rho x}{|x|} &\mbox{ if } |x|\geq \rho.
\end{array}
\right.
\end{equation*}
We define the barycenter map $\beta_{\e}: \N_{\e}\rightarrow \R^{3}$ as follows
\begin{align*}
\beta_{\e}(u)=\frac{\displaystyle{\int_{\R^{3}} \varUpsilon(\e x)|u(x)|^{4} \,dx}}{\displaystyle{\int_{\R^{3}} |u(x)|^{4} \,dx}}.
\end{align*}

\noindent
Arguing as Lemma $4.3$ in \cite{AD}, it is easy to see that the function $\beta_{\e}$ verifies the following limit:
\begin{lem}\label{lem3.5N}
\begin{equation*}
\lim_{\e \rightarrow 0} \beta_{\e}(\Phi_{\e}(y))=y \mbox{ uniformly in } y\in M.
\end{equation*}
\end{lem}

\noindent
At this point, we introduce a subset $\widetilde{\N}_{\e}$ of $\N_{\e}$ by taking a function $h_{1}:\R^{+}\rightarrow \R^{+}$ such that $h_{1}(\e)\rightarrow 0$ as $\e \rightarrow 0$, and setting
$$
\widetilde{\N}_{\e}=\left \{u\in \N_{\e}: J_{\e}(u)\leq c_{V_{0}}+h_{1}(\e)\right\}.
$$
Fixed $y\in M$, from Lemma \ref{lem3.4} follows that $h_{1}(\e)=|J_{\e}(\Phi_{\e}(y))-c_{V_{0}}|\rightarrow 0$ as $\e \rightarrow 0$. Therefore $\Phi_{\e}(y)\in \widetilde{\N}_{\e}$, and $\widetilde{\N}_{\e}\neq \emptyset$ for any $\e>0$. Moreover, proceeding as in Lemma $4.5$ in \cite{AD}, we have:
\begin{lem}\label{lem3.5}
$$
\lim_{\e \rightarrow 0} \sup_{u\in \widetilde{\mathcal{N}}_{\e}} {\rm dist}(\beta_{\e}(u), M_{\delta})=0.
$$
\end{lem}

We conclude this section giving the proof of our multiplicity result for \eqref{MPe}.
\begin{thm}\label{multiple}
For any $\delta>0$ such that $M_{\delta}\subset \Lambda$, there exists $\tilde{\e}_{\delta}>0$ such that, for any $\e\in (0, \e_{\delta})$, problem \eqref{MPe} has at least $cat_{M_{\delta}}(M)$ nontrivial solutions.
\end{thm}
\begin{proof}
Given  $\delta>0$ such that $M_{\delta}\subset \Lambda$, we can use Lemma \ref{lem3.5N}, Lemma \ref{lem3.4}, Lemma \ref{lem3.5} and argue as in \cite{CL} to deduce the existence of $\tilde{\e}_{\delta}>0$ such that, for any $\e\in (0, \e_{\delta})$, the following diagram
$$
M \stackrel{\Phi_{\e}}\rightarrow \widetilde{\mathcal{N}}_{\e} \stackrel{\beta_{\e}}\rightarrow M_{\delta}
$$
is well defined and $\beta_{\e}\circ \Phi_{\e}$ is homotopically equivalent to the embedding  $\iota: M\rightarrow M_{\delta}$. Thus $cat_{\widetilde{\mathcal{N}}_{\e}}(\widetilde{\mathcal{N}}_{\e})\geq cat_{M_{\delta}}(M)$.
It follows from Proposition \ref{propPSc} and standard Ljusternik-Schnirelmann theory that $J_{\e}$  possesses at least $cat_{\widetilde{\mathcal{N}}_{\e}}(\widetilde{\mathcal{N}}_{\e})$  critical points on $\mathcal{N}_{\e}$. Using Corollary \ref{cor}  we can obtain $cat_{M_{\delta}}(M)$  nontrivial solutions for  \eqref{MPe}.
\end{proof}

\section{Proof of Theorem \ref{thm1}}\label{sec5}
In this last section we provide the proof of our main result.
Firstly, we develop a Moser iteration scheme \cite{Moser} which will be the main key to deduce that the solutions to \eqref{Pe} are indeed solutions  to \eqref{P}.
\begin{lem}\label{moser} 
Let $\e_{n}\rightarrow 0$ and $u_{n}\in \widetilde{\mathcal{N}}_{\e_{n}}$ be a solution to \eqref{MPe}. 
Then $v_{n}=|u_{n}|(\cdot+\tilde{y}_{n})$ satisfies $v_{n}\in L^{\infty}(\R^{3},\R)$ and there exists $C>0$ such that 
$$
\|v_{n}\|_{L^{\infty}(\R^{3})}\leq C \mbox{ for all } n\in \mathbb{N},
$$
where $\tilde{y}_{n}$ is given by Lemma \ref{prop3.3}.
Moreover
$$
\lim_{|x|\rightarrow \infty} v_{n}(x)=0 \mbox{ uniformly in } n\in \mathbb{N}.
$$
\end{lem}
\begin{proof}
For any $L>0$ we define $u_{L,n}:=\min\{|u_{n}|, L\}\geq 0$ and we set $v_{L, n}=u_{L,n}^{2(\beta-1)}u_{n}$ where $\beta>1$ will be chosen after \eqref{simo1}.
Taking $v_{L, n}$ as a test function in (\ref{MPe}) we can see that
\begin{align}\label{conto1FF}
&\Re\left(\iint_{\R^{6}} \frac{(u_{n}(x)-u_{n}(y)e^{\imath A(\frac{x+y}{2})\cdot (x-y)})}{|x-y|^{3+2s}} \overline{(u_{n}u_{L,n}^{2(\beta-1)}(x)-u_{n}u_{L,n}^{2(\beta-1)}(y)e^{\imath A(\frac{x+y}{2})\cdot (x-y)})} \, dx dy\right)   \nonumber \\
&=-\int_{\R^{3}}\phi_{|u_{n}|}^{t}|u_{n}|^{2}u_{L,n}^{2(\beta-1)}dx+\int_{\R^{3}} g_{\e_{n}}(x, |u_{n}|^{2}) |u_{n}|^{2}u_{L,n}^{2(\beta-1)}  \,dx-\int_{\R^{3}} V_{\e_{n}}(x) |u_{n}|^{2} u_{L,n}^{2(\beta-1)} \, dx.
\end{align}
Let us note that
\begin{align*}
&\Re\left[(u_{n}(x)-u_{n}(y)e^{\imath A(\frac{x+y}{2})\cdot (x-y)})\overline{(u_{n}u_{L,n}^{2(\beta-1)}(x)-u_{n}u_{L,n}^{2(\beta-1)}(y)e^{\imath A(\frac{x+y}{2})\cdot (x-y)})}\right] \\
&=\Re\Bigl[|u_{n}(x)|^{2}v_{L}^{2(\beta-1)}(x)-u_{n}(x)\overline{u_{n}(y)} u_{L,n}^{2(\beta-1)}(y)e^{-\imath A(\frac{x+y}{2})\cdot (x-y)}-u_{n}(y)\overline{u_{n}(x)} u_{L,n}^{2(\beta-1)}(x) e^{\imath A(\frac{x+y}{2})\cdot (x-y)} \\
&+|u_{n}(y)|^{2}u_{L,n}^{2(\beta-1)}(y) \Bigr] \\
&\geq (|u_{n}(x)|^{2}u_{L,n}^{2(\beta-1)}(x)-|u_{n}(x)||u_{n}(y)|u_{L,n}^{2(\beta-1)}(y)-|u_{n}(y)||u_{n}(x)|u_{L,n}^{2(\beta-1)}(x)+|u_{n}(y)|^{2}u^{2(\beta-1)}_{L,n}(y) \\
&=(|u_{n}(x)|-|u_{n}(y)|)(|u_{n}(x)|u_{L,n}^{2(\beta-1)}(x)-|u_{n}(y)|u_{L,n}^{2(\beta-1)}(y)),
\end{align*}
so we have
\begin{align}\label{realeF}
&\Re\left(\iint_{\R^{6}} \frac{(u_{n}(x)-u_{n}(y)e^{\imath A(\frac{x+y}{2})\cdot (x-y)})}{|x-y|^{3+2s}} \overline{(u_{n}u_{L,n}^{2(\beta-1)}(x)-u_{n}u_{L,n}^{2(\beta-1)}(y)e^{\imath A(\frac{x+y}{2})\cdot (x-y)})} \, dx dy\right) \nonumber\\
&\geq \iint_{\R^{6}} \frac{(|u_{n}(x)|-|u_{n}(y)|)}{|x-y|^{3+2s}} (|u_{n}(x)|u_{L,n}^{2(\beta-1)}(x)-|u_{n}(y)|u_{L,n}^{2(\beta-1)}(y))\, dx dy.
\end{align}
For all $t\geq 0$, let us define
\begin{equation*}
\gamma(t)=\gamma_{L, \beta}(t)=t t_{L}^{2(\beta-1)}
\end{equation*}
where  $t_{L}=\min\{t, L\}$. 
Since $\gamma$ is an increasing function, we have
\begin{align*}
(a-b)(\gamma(a)- \gamma(b))\geq 0 \quad \mbox{ for any } a, b\in \R.
\end{align*}
Let us define the functions 
\begin{equation*}
\Lambda(t)=\frac{|t|^{2}}{2} \quad \mbox{ and } \quad \Gamma(t)=\int_{0}^{t} (\gamma'(\tau))^{\frac{1}{2}} d\tau. 
\end{equation*}
and we note that
\begin{equation}\label{Gg}
\Lambda'(a-b)(\gamma(a)-\gamma(b))\geq |\Gamma(a)-\Gamma(b)|^{2} \mbox{ for any } a, b\in\R. 
\end{equation}
Indeed, for any $a, b\in \R$ such that $a<b$, the Jensen inequality yields
\begin{align*}
\Lambda'(a-b)(\gamma(a)-\gamma(b))&=(a-b)\int_{b}^{a} \gamma'(t)dt\\
&=(a-b)\int_{b}^{a} (\Gamma'(t))^{2}dt \\
&\geq \left(\int_{b}^{a} \Gamma'(t) dt\right)^{2}\\
&=(\Gamma(a)-\Gamma(b))^{2}.
\end{align*}
In similar fashion we can prove that if $a\geq b$  then $\Lambda'(a-b)(\gamma(a)-\gamma(b))\geq (\Gamma(b)-\Gamma(a))^{2}$ that is \eqref{Gg} holds.
Then, in view of \eqref{Gg}, we can see that
\begin{align}\label{Gg1}
|\Gamma(|u_{n}(x)|)- \Gamma(|u_{n}(y)|)|^{2} \leq (|u_{n}(x)|- |u_{n}(y)|)((|u_{n}|u_{L,n}^{2(\beta-1)})(x)- (|u_{n}|u_{L,n}^{2(\beta-1)})(y)). 
\end{align}
Taking into account \eqref{realeF} and \eqref{Gg1}, we obtain
\begin{align}\label{conto1FFF}
\Re\left(\iint_{\R^{6}} \frac{(u_{n}(x)-u_{n}(y)e^{\imath A(\frac{x+y}{2})\cdot (x-y)})}{|x-y|^{3+2s}} \overline{(u_{n}u_{L,n}^{2(\beta-1)}(x)-u_{n}u_{L,n}^{2(\beta-1)}(y)e^{\imath A(\frac{x+y}{2})\cdot (x-y)})} \, dx dy\right) 
\geq [\Gamma(|u_{n}|)]^{2}.
\end{align}
Since $\Gamma(|u_{n}|)\geq \frac{1}{\beta} |u_{n}| u_{L,n}^{\beta-1}$ and using the fractional Sobolev embedding $\mathcal{D}^{s,2}(\R^{3}, \R)\subset L^{\2}(\R^{3}, \R)$ (see \cite{DPV}), we deduce that 
\begin{equation}\label{SS1}
[\Gamma(|u_{n}|)]^{2}\geq S_{*} \|\Gamma(|u_{n}|)\|^{2}_{L^{\2}(\R^{3})}\geq \left(\frac{1}{\beta}\right)^{2} S_{*}\||u_{n}| u_{L,n}^{\beta-1}\|^{2}_{L^{\2}(\R^{3})}.
\end{equation}
Putting together \eqref{conto1FF}, \eqref{conto1FFF}, \eqref{SS1} and using $(4)$ of Lemma \ref{poisson}, we can infer that
\begin{align}\label{BMS}
\left(\frac{1}{\beta}\right)^{2} S_{*}\||u_{n}| u_{L,n}^{\beta-1}\|^{2}_{L^{\2}(\R^{3})}+\int_{\R^{3}} V_{\e_{n}}(x)|u_{n}|^{2}u_{L,n}^{2(\beta-1)} dx\leq \int_{\R^{3}} g_{\e_{n}}(x, |u_{n}|^{2}) |u_{n}|^{2} u_{L,n}^{2(\beta-1)} dx.
\end{align}
On the other hand, from assumptions $(g_1)$ and $(g_2)$, for any $\xi>0$ there exists $C_{\xi}>0$ such that
\begin{equation}\label{SS2}
g_{\e_{n}}(x, t^{2})t^{2}\leq \xi |t|^{2}+C_{\xi}|t|^{\2} \mbox{ for all } t\in \R.
\end{equation}
Taking $\xi\in (0, V_{0})$ and using \eqref{BMS} and \eqref{SS2} we can see that
\begin{equation}\label{simo1}
\|w_{L,n}\|_{L^{\2}(\R^{3})}^{2}\leq C \beta^{2} \int_{\R^{3}} |u_{n}|^{\2}u_{L,n}^{2(\beta-1)},
\end{equation}
where $w_{L,n}:=|u_{n}| u_{L,n}^{\beta-1}$.\\
Now, we take $\beta=\frac{\2}{2}$ and fix $R>0$. Recalling that $0\leq u_{L,n}\leq |u_{n}|$ and applying H\"older inequality we have
\begin{align}\label{simo2}
\int_{\R^{3}} |u_{n}|^{\2}u_{L,n}^{2(\beta-1)}dx&=\int_{\R^{3}} |u_{n}|^{\2-2} |u_{n}|^{2} u_{L,n}^{\2-2}dx \nonumber\\
&=\int_{\R^{3}} |u_{n}|^{\2-2} (|u_{n}| u_{L,n}^{\frac{\2-2}{2}})^{2}dx \nonumber\\
&\leq \int_{\{|u_{n}|<R\}} R^{\2-2} |u_{n}|^{\2} dx+\int_{\{|u_{n}|>R\}} |u_{n}|^{\2-2} (|u_{n}| u_{L,n}^{\frac{\2-2}{2}})^{2}dx \nonumber\\
&\leq \int_{\{|u_{n}|<R\}} R^{\2-2} |u_{n}|^{\2} dx+\left(\int_{\{|u_{n}|>R\}} |u_{n}|^{\2} dx\right)^{\frac{\2-2}{\2}} \left(\int_{\R^{3}} (|u_{n}| u_{L,n}^{\frac{\2-2}{2}})^{\2}dx\right)^{\frac{2}{\2}}.
\end{align}
Since $(|u_{n}|)$ is bounded in $H^{s}(\R^{3}, \R)$, we can see that for any $R$ sufficiently large
\begin{equation}\label{simo3}
\left(\int_{\{|u_{n}|>R\}} |u_{n}|^{\2} dx\right)^{\frac{\2-2}{\2}}\leq  \frac{1}{2\beta^{2}}.
\end{equation}
Putting together \eqref{simo1}, \eqref{simo2} and \eqref{simo3} we get
\begin{equation*}
\left(\int_{\R^{3}} (|u_{n}| u_{L,n}^{\frac{\2-2}{2}})^{\2} \right)^{\frac{2}{\2}}\leq C\beta^{2} \int_{\R^{3}} R^{\2-2} |u_{n}|^{\2} dx<\infty
\end{equation*}
and taking the limit as $L\rightarrow \infty$ we obtain $|u_{n}|\in L^{\frac{(\2)^{2}}{2}}(\R^{3},\R)$.

Now, using $0\leq u_{L,n}\leq |u_{n}|$ and passing to the limit as $L\rightarrow \infty$ in \eqref{simo1} we have
\begin{equation*}
\|u_{n}\|_{L^{\beta\2}(\R^{3})}^{2\beta}\leq C \beta^{2} \int_{\R^{3}} |u_{n}|^{\2+2(\beta-1)},
\end{equation*}
from which we deduce that
\begin{equation*}
\left(\int_{\R^{3}} |u_{n}|^{\beta\2} dx\right)^{\frac{1}{(\beta-1)\2}}\leq C \beta^{\frac{1}{\beta-1}} \left(\int_{\R^{3}} |u_{n}|^{\2+2(\beta-1)}\right)^{\frac{1}{2(\beta-1)}}.
\end{equation*}
For $m\geq 1$ we define $\beta_{m+1}$ inductively so that $\2+2(\beta_{m+1}-1)=\2 \beta_{m}$ and $\beta_{1}=\frac{\2}{2}$. Then we have
\begin{equation*}
\left(\int_{\R^{3}} |u_{n}|^{\beta_{m+1}\2} dx\right)^{\frac{1}{(\beta_{m+1}-1)\2}}\leq C \beta_{m+1}^{\frac{1}{\beta_{m+1}-1}} \left(\int_{\R^{3}} |u_{n}|^{\2\beta_{m}}\right)^{\frac{1}{\2(\beta_{m}-1)}}.
\end{equation*}
Let us define
$$
D_{m}=\left(\int_{\R^{3}} |u_{n}|^{\2\beta_{m}}\right)^{\frac{1}{\2(\beta_{m}-1)}}.
$$
Using an iteration argument, we can find $C_{0}>0$ independent of $m$ such that 
$$
D_{m+1}\leq \prod_{k=1}^{m} C \beta_{k+1}^{\frac{1}{\beta_{k+1}-1}}  D_{1}\leq C_{0} D_{1}.
$$
Taking the limit as $m\rightarrow \infty$ we get 
\begin{equation}\label{UBu}
\|u_{n}\|_{L^{\infty}(\R^{3})}\leq C_{0}D_{1}=:K \mbox{ for all } n\in \mathbb{N}.
\end{equation}
Moreover, by interpolation, $(|u_{n}|)$ strongly converges in $L^{r}(\R^{3}, \R)$ for all $r\in (2, \infty)$, and in view of the growth assumptions on $g$, also $g(\e x, |u_{n}|^{2})|u_{n}|$ strongly converges  in the same Lebesgue spaces. \\
Now, we aim to prove that $|u_{n}|$ is a weak subsolution to 
\begin{equation}\label{Kato0}
\left\{
\begin{array}{ll}
(-\Delta)^{s}v+V_{0} v=g_{\e_{n}}(x, v^{2})v &\mbox{ in } \R^{3} \\
v\geq 0 \quad \mbox{ in } \R^{3}.
\end{array}
\right.
\end{equation}
In some sense, we are going to prove that a Kato's inequality holds for the modified problem \eqref{MPe}.
Fix $\varphi\in C^{\infty}_{c}(\R^{3}, \R)$ such that $\varphi\geq 0$, and we take $\psi_{\delta, n}=\frac{u_{n}}{u_{\delta, n}}\varphi$ as test function in \eqref{Pe}, where we set $u_{\delta,n}=\sqrt{|u_{n}|^{2}+\delta^{2}}$ for $\delta>0$. We note that $\psi_{\delta, n}\in H^{s}_{\e_{n}}$ for all $\delta>0$ and $n\in \mathbb{N}$. Indeed $\int_{\R^{3}} V_{\e_{n}}(x) |\psi_{\delta,n}|^{2} dx\leq \int_{\supp(\varphi)} V_{\e_{n}}(x)\varphi^{2} dx<\infty$. 
On the other hand, we can observe
\begin{align*}
\psi_{\delta,n}(x)-\psi_{\delta,n}(y)e^{\imath A_{\e}(\frac{x+y}{2})\cdot (x-y)}&=\left(\frac{u_{n}(x)}{u_{\delta,n}(x)}\right)\varphi(x)-\left(\frac{u_{n}(y)}{u_{\delta,n}(y)}\right)\varphi(y)e^{\imath A_{\e}(\frac{x+y}{2})\cdot (x-y)}\\
&=\left[\left(\frac{u_{n}(x)}{u_{\delta,n}(x)}\right)-\left(\frac{u_{n}(y)}{u_{\delta,n}(x)}\right)e^{\imath A_{\e}(\frac{x+y}{2})\cdot (x-y)}\right]\varphi(x) \\
&+\left[\varphi(x)-\varphi(y)\right] \left(\frac{u_{n}(y)}{u_{\delta,n}(x)}\right) e^{\imath A_{\e}(\frac{x+y}{2})\cdot (x-y)} \\
&+\left(\frac{u_{n}(y)}{u_{\delta,n}(x)}-\frac{u_{n}(y)}{u_{\delta,n}(y)}\right)\varphi(y) e^{\imath A_{\e}(\frac{x+y}{2})\cdot (x-y)}
\end{align*}
which implies that
\begin{align*}
&|\psi_{\delta,n}(x)-\psi_{\delta,n}(y)e^{\imath A_{\e}(\frac{x+y}{2})\cdot (x-y)}|^{2} \\
&\leq \frac{4}{\delta^{2}}|u_{n}(x)-u_{n}(y)e^{\imath A_{\e}(\frac{x+y}{2})\cdot (x-y)}|^{2}\|\varphi\|^{2}_{L^{\infty}(\R^{3})} +\frac{4}{\delta^{2}}|\varphi(x)-\varphi(y)|^{2} \||u_{n}|\|^{2}_{L^{\infty}(\R^{3})} \\
&+\frac{4}{\delta^{4}} \||u_{n}|\|^{2}_{L^{\infty}(\R^{3})} \|\varphi\|^{2}_{L^{\infty}(\R^{3})} |u_{\delta,n}(y)-u_{\delta,n}(x)|^{2} \\
&\leq \frac{4}{\delta^{2}}|u_{n}(x)-u_{n}(y)e^{\imath A_{\e}(\frac{x+y}{2})\cdot (x-y)}|^{2}\|\varphi\|^{2}_{L^{\infty}(\R^{3})} +\frac{4K^{2}}{\delta^{2}}|\varphi(x)-\varphi(y)|^{2} \\
&+\frac{4K^{2}}{\delta^{4}} \|\varphi\|^{2}_{L^{\infty}(\R^{3})} ||u_{n}(y)|-|u_{n}(x)||^{2} 
\end{align*}
where we used $|z+w+k|^{2}\leq 4(|z|^{2}+|w|^{2}+|k|^{2})$ for all $z,w,k\in \C$, $|e^{\imath t}|=1$ for all $t\in \R$, $u_{\delta,n}\geq \delta$, $|\frac{u_{n}}{u_{\delta,n}}|\leq 1$, \eqref{UBu} and $|\sqrt{|z|^{2}+\delta^{2}}-\sqrt{|w|^{2}+\delta^{2}}|\leq ||z|-|w||$ for all $z, w\in \C$.\\
Since $u_{n}\in H^{s}_{\e_{n}}$, $|u_{n}|\in H^{s}(\R^{3}, \R)$ (by Lemma \ref{DI}) and $\varphi\in C^{\infty}_{c}(\R^{3}, \R)$, we deduce that $\psi_{\delta,n}\in H^{s}_{\e_{n}}$.

Then we have
\begin{align}\label{Kato1}
&\Re\left[\iint_{\R^{6}} \frac{(u_{n}(x)-u_{n}(y)e^{\imath A_{\e}(\frac{x+y}{2})\cdot (x-y)})}{|x-y|^{3+2s}} \left(\frac{\overline{u_{n}(x)}}{u_{\delta,n}(x)}\varphi(x)-\frac{\overline{u_{n}(y)}}{u_{\delta,n}(y)}\varphi(y)e^{-\imath A_{\e}(\frac{x+y}{2})\cdot (x-y)}  \right) dx dy\right] \nonumber\\
&+\int_{\R^{3}} V_{\e_{n}}(x)\frac{|u_{n}|^{2}}{u_{\delta,n}}\varphi dx+\int_{\R^{3}} \phi_{|u_{n}|}^{t} \frac{|u_{n}|^{2}}{u_{\delta,n}}\varphi dx=\int_{\R^{3}} g_{\e_{n}}(x, |u_{n}|^{2})\frac{|u_{n}|^{2}}{u_{\delta,n}}\varphi dx.
\end{align}
Now, using $\Re(z)\leq |z|$ for all $z\in \C$ and  $|e^{\imath t}|=1$ for all $t\in \R$, we have
\begin{align}\label{alves1}
&\Re\left[(u_{n}(x)-u_{n}(y)e^{\imath A_{\e}(\frac{x+y}{2})\cdot (x-y)}) \left(\frac{\overline{u_{n}(x)}}{u_{\delta,n}(x)}\varphi(x)-\frac{\overline{u_{n}(y)}}{u_{\delta,n}(y)}\varphi(y)e^{-\imath A_{\e}(\frac{x+y}{2})\cdot (x-y)}  \right)\right] \nonumber\\
&=\Re\left[\frac{|u_{n}(x)|^{2}}{u_{\delta,n}(x)}\varphi(x)+\frac{|u_{n}(y)|^{2}}{u_{\delta,n}(y)}\varphi(y)-\frac{u_{n}(x)\overline{u_{n}(y)}}{u_{\delta,n}(y)}\varphi(y)e^{-\imath A_{\e}(\frac{x+y}{2})\cdot (x-y)} -\frac{u_{n}(y)\overline{u_{n}(x)}}{u_{\delta,n}(x)}\varphi(x)e^{\imath A_{\e}(\frac{x+y}{2})\cdot (x-y)}\right] \nonumber \\
&\geq \left[\frac{|u_{n}(x)|^{2}}{u_{\delta,n}(x)}\varphi(x)+\frac{|u_{n}(y)|^{2}}{u_{\delta,n}(y)}\varphi(y)-|u_{n}(x)|\frac{|u_{n}(y)|}{u_{\delta,n}(y)}\varphi(y)-|u_{n}(y)|\frac{|u_{n}(x)|}{u_{\delta,n}(x)}\varphi(x) \right].
\end{align}
Let us note that
\begin{align}\label{alves2}
&\frac{|u_{n}(x)|^{2}}{u_{\delta,n}(x)}\varphi(x)+\frac{|u_{n}(y)|^{2}}{u_{\delta,n}(y)}\varphi(y)-|u_{n}(x)|\frac{|u_{n}(y)|}{u_{\delta,n}(y)}\varphi(y)-|u_{n}(y)|\frac{|u_{n}(x)|}{u_{\delta,n}(x)}\varphi(x) \nonumber\\
&=  \frac{|u_{n}(x)|}{u_{\delta,n}(x)}(|u_{n}(x)|-|u_{n}(y)|)\varphi(x)-\frac{|u_{n}(y)|}{u_{\delta,n}(y)}(|u_{n}(x)|-|u_{n}(y)|)\varphi(y) \nonumber\\
&=\left[\frac{|u_{n}(x)|}{u_{\delta,n}(x)}(|u_{n}(x)|-|u_{n}(y)|)\varphi(x)-\frac{|u_{n}(x)|}{u_{\delta,n}(x)}(|u_{n}(x)|-|u_{n}(y)|)\varphi(y)\right] \nonumber\\
&+\left(\frac{|u_{n}(x)|}{u_{\delta,n}(x)}-\frac{|u_{n}(y)|}{u_{\delta,n}(y)} \right) (|u_{n}(x)|-|u_{n}(y)|)\varphi(y) \nonumber\\
&=\frac{|u_{n}(x)|}{u_{\delta,n}(x)}(|u_{n}(x)|-|u_{n}(y)|)(\varphi(x)-\varphi(y)) +\left(\frac{|u_{n}(x)|}{u_{\delta,n}(x)}-\frac{|u_{n}(y)|}{u_{\delta,n}(y)} \right) (|u_{n}(x)|-|u_{n}(y)|)\varphi(y) \nonumber\\
&\geq \frac{|u_{n}(x)|}{u_{\delta,n}(x)}(|u_{n}(x)|-|u_{n}(y)|)(\varphi(x)-\varphi(y)) 
\end{align}
where in the last inequality we used the fact that
$$
\left(\frac{|u_{n}(x)|}{u_{\delta,n}(x)}-\frac{|u_{n}(y)|}{u_{\delta,n}(y)} \right) (|u_{n}(x)|-|u_{n}(y)|)\varphi(y)\geq 0
$$
because
$$
h(t)=\frac{t}{\sqrt{t^{2}+\delta^{2}}} \mbox{ is increasing for } t\geq 0 \quad \mbox{ and } \quad \varphi\geq 0 \mbox{ in }\R^{3}.
$$
Since
$$
\frac{|\frac{|u_{n}(x)|}{u_{\delta,n}(x)}(|u_{n}(x)|-|u_{n}(y)|)(\varphi(x)-\varphi(y))|}{|x-y|^{3+2s}}\leq \frac{||u_{n}(x)|-|u_{n}(y)||}{|x-y|^{\frac{3+2s}{2}}} \frac{|\varphi(x)-\varphi(y)|}{|x-y|^{\frac{3+2s}{2}}}\in L^{1}(\R^{6}),
$$
and $\frac{|u_{n}(x)|}{u_{\delta,n}(x)}\rightarrow 1$ a.e. in $\R^{3}$ as $\delta\rightarrow 0$,
we can use \eqref{alves1}, \eqref{alves2} and the Dominated Convergence Theorem to deduce that
\begin{align}\label{Kato2}
&\limsup_{\delta\rightarrow 0} \Re\left[\iint_{\R^{6}} \frac{(u_{n}(x)-u_{n}(y)e^{\imath A_{\e}(\frac{x+y}{2})\cdot (x-y)})}{|x-y|^{3+2s}} \left(\frac{\overline{u_{n}(x)}}{u_{\delta,n}(x)}\varphi(x)-\frac{\overline{u_{n}(y)}}{u_{\delta,n}(y)}\varphi(y)e^{-\imath A_{\e}(\frac{x+y}{2})\cdot (x-y)}  \right) dx dy\right] \nonumber\\
&\geq \limsup_{\delta\rightarrow 0} \iint_{\R^{6}} \frac{|u_{n}(x)|}{u_{\delta,n}(x)}(|u_{n}(x)|-|u_{n}(y)|)(\varphi(x)-\varphi(y)) \frac{dx dy}{|x-y|^{3+2s}} \nonumber\\
&=\iint_{\R^{6}} \frac{(|u_{n}(x)|-|u_{n}(y)|)(\varphi(x)-\varphi(y))}{|x-y|^{3+2s}} dx dy.
\end{align}
On the other hand, from the Dominated Convergence Theorem again (we recall that $\frac{|u_{n}|^{2}}{u_{\delta, n}}\leq |u_{n}|$), Fatou's Lemma and $\varphi\in C^{\infty}_{c}(\R^{3}, \R)$ we can see that
\begin{equation}\label{Kato3}
\lim_{\delta\rightarrow 0} \int_{\R^{3}} V_{\e_{n}}(x)\frac{|u_{n}|^{2}}{u_{\delta,n}}\varphi dx=\int_{\R^{3}} V_{\e_{n}}(x)|u_{n}|\varphi dx\geq \int_{\R^{3}} V_{0}|u_{n}|\varphi dx
\end{equation}
\begin{equation}\label{KatoP}
\liminf_{\delta\rightarrow 0} \int_{\R^{3}} \phi_{|u_{n}|}^{t} \frac{|u_{n}|^{2}}{u_{\delta,n}}\varphi dx\geq \int_{\R^{3}} \phi_{|u|}^{t} |u|\varphi dx\geq 0
\end{equation}
and
\begin{equation}\label{Kato4}
\lim_{\delta\rightarrow 0}  \int_{\R^{3}} g_{\e_{n}}(x, |u_{n}|^{2})\frac{|u_{n}|^{2}}{u_{\delta,n}}\varphi dx=\int_{\R^{3}} g_{\e_{n}}(x, |u_{n}|^{2}) |u_{n}|\varphi dx.
\end{equation}
Putting together \eqref{Kato1}, \eqref{Kato2}, \eqref{KatoP}, \eqref{Kato3} and \eqref{Kato4} we can deduce that
\begin{align*}
\iint_{\R^{6}} \frac{(|u_{n}(x)|-|u_{n}(y)|)(\varphi(x)-\varphi(y))}{|x-y|^{3+2s}} dx dy+\int_{\R^{3}} V_{0}|u_{n}|\varphi dx\leq 
\int_{\R^{3}} g_{\e_{n}}(x, |u_{n}|^{2}) |u_{n}|\varphi dx
\end{align*}
for any $\varphi\in C^{\infty}_{c}(\R^{3}, \R)$ such that $\varphi\geq 0$, that is $|u_{n}|$ is a weak subsolution to \eqref{Kato0}.
Now, it is clear that $v_{n}=|u_{n}|(\cdot+\tilde{y}_{n})$ solves
\begin{equation}\label{Pkat}
(-\Delta)^{s} v_{n} + V_{0}v_{n}\leq g(\e_{n} x+\e_{n}\tilde{y}_{n}, v_{n}^{2})v_{n} \mbox{ in } \R^{3}. 
\end{equation}
Let us denote by $z_{n}\in H^{s}(\R^{3}, \R)$ the unique solution to
\begin{equation}\label{US}
(-\Delta)^{s} z_{n} + V_{0}z_{n}=g_{n} \mbox{ in } \R^{3},
\end{equation}
where
$$
g_{n}:=g(\e_{n} x+\e_{n}\tilde{y}_{n}, v_{n}^{2})v_{n}\in L^{r}(\R^{3}, \R) \quad \forall r\in [2, \infty].
$$
Since \eqref{UBu} yields $\|v_{n}\|_{L^{\infty}(\R^{3})}\leq C$ for all $n\in \mathbb{N}$, by interpolation we know that $v_{n}\rightarrow v$ strongly converges in $L^{r}(\R^{3}, \R)$ for all $r\in (2, \infty)$, for some $v\in L^{r}(\R^{3}, \R)$, and from the growth assumptions on $f$, we can see that also $g_{n}\rightarrow  f(v^{2})v$ in $L^{r}(\R^{3}, \R)$ and $\|g_{n}\|_{L^{\infty}(\R^{3})}\leq C$ for all $n\in \mathbb{N}$.
In view of \cite{FQT}, we deduce that $z_{n}=\mathcal{K}*g_{n}$, where $\mathcal{K}$ is the Bessel kernel, and arguing as in \cite{AM}, we obtain that $|z_{n}(x)|\rightarrow 0$ as $|x|\rightarrow \infty$ uniformly with respect to $n\in \mathbb{N}$.
Since $v_{n}$ satisfies \eqref{Pkat} and $z_{n}$ solves \eqref{US}, by comparison it is easy to see that $0\leq v_{n}\leq z_{n}$ a.e. in $\R^{3}$ and for all $n\in \mathbb{N}$. Then we can conclude that $v_{n}(x)\rightarrow 0$ as $|x|\rightarrow \infty$ uniformly with respect to $n\in \mathbb{N}$.
\end{proof}

\noindent
Now we are ready to give the proof of the main result of this paper.
\begin{proof}[Proof of Theorem \ref{thm1}]
Let $\delta>0$ be such that $M_{\delta}\subset \Lambda$, and we show that there exists  $\hat{\e}_{\delta}>0$ such that for any $\e\in (0, \hat{\e}_{\delta})$ and any solution $u_{\e}\in \widetilde{\mathcal{N}}_{\e}$ of \eqref{MPe} we have
\begin{equation}\label{Ua}
\|u_{\e}\|_{L^{\infty}(\R^{3}\setminus \Lambda_{\e})}<t_{a}.
\end{equation}
Assume by contradiction that for some sequence $\e_{n}\rightarrow 0$ we can obtain  $u_{n}:=u_{\e_{n}}\in \widetilde{\mathcal{N}}_{\e_{n}}$ such that
\begin{equation}\label{AbsAFF}
\|u_{n}\|_{L^{\infty}(\R^{3}\setminus \Lambda_{\e})}\geq t_{a}.
\end{equation}
Since $J_{\e_{n}}(u_{n})\leq c_{V_{0}}+h_{1}(\e_{n})$, we can argue as in the first part of Lemma \ref{prop3.3} to see that $J_{\e_{n}}(u_{n})\rightarrow c_{V_{0}}$.
Using Lemma \ref{prop3.3} there exists $(\tilde{y}_{n})\subset \R^{3}$ such that $\e_{n}\tilde{y}_{n}\rightarrow y_{0}$ for some $y_{0} \in M$. 
Now, we can find $r>0$ such that, for some subsequence still denoted by itself, we obtain $B_{r}(\tilde{y}_{n})\subset \Lambda$ for all $n\in \mathbb{N}$.
Therefore $B_{\frac{r}{\e_{n}}}(\tilde{y}_{n})\subset \Lambda_{\e_{n}}$ $n\in \mathbb{N}$. As a consequence 
$$
\R^{3}\setminus \Lambda_{\e_{n}}\subset \R^{3} \setminus B_{\frac{r}{\e_{n}}}(\tilde{y}_{n}) \mbox{ for any } n\in \mathbb{N}.
$$ 
In view of Lemma \ref{moser}, there exists $R>0$ such that 
$$
v_{n}(x)<t_{a} \mbox{ for } |x|\geq R, n\in \mathbb{N},
$$ 
where $v_{n}(x)=|u_{n}|(x+ \tilde{y}_{n})$. 
Hence $|u_{n}(x)|<t_{a}$ for any $x\in \R^{3}\setminus B_{R}(\tilde{y}_{n})$ and $n\in \mathbb{N}$. Then there exists $\nu \in \mathbb{N}$ such that for any $n\geq \nu$ and $r/\e_{n}>R$ it holds 
$$
\R^{3}\setminus \Lambda_{\e_{n}}\subset \R^{3} \setminus B_{\frac{r}{\e_{n}}}(\tilde{y}_{n})\subset \R^{3}\setminus B_{R}(\tilde{y}_{n}).
$$ 
Then $|u_{n}(x)|<t_{a}$ for any $x\in \R^{3}\setminus \Lambda_{\e_{n}}$ and $n\geq \nu$, and this contradicts \eqref{AbsAFF}.

Let $\tilde{\e}_{\delta}>0$ be given by Theorem \ref{multiple} and we set $\e_{\delta}=\min\{\tilde{\e}_{\delta}, \hat{\e}_{\delta} \}$. Applying Theorem \ref{multiple} we obtain $cat_{M_{\delta}}(M)$ nontrivial solutions to \eqref{MPe}.
If $u\in \h$ is one of these solutions, then $u\in \widetilde{\mathcal{N}}_{\e}$, and in view of \eqref{Ua} and the definition of $g$ we can infer that $u$ is also a solution to \eqref{MPe}. Observing that $\hat{u}_{\e}(x)=u_{\e}(x/\e)$ is a solution to (\ref{P}), we can deduce that \eqref{P} has at least $cat_{M_{\delta}}(M)$ nontrivial solutions.
Finally, we study the behavior of the maximum points of  $|\hat{u}_{n}|$. Take $\e_{n}\rightarrow 0$ and $(u_{n})$ a sequence of solutions to \eqref{MPe}. In view of $(g_1)$, there exists $\gamma\in (0,t_{a})$ such that
\begin{align}\label{4.18HZ}
g_{\e}(x, t^{2})t^{2}\leq \frac{V_{0}}{2}t^{2}, \mbox{ for all } x\in \R^{3}, |t|\leq \gamma.
\end{align}
Using a similar discussion as above, we can take $R>0$ such that
\begin{align}\label{4.19HZ}
\|u_{n}\|_{L^{\infty}(B^{c}_{R}(\tilde{y}_{n}))}<\gamma.
\end{align}
Up to a subsequence, we may also assume that
\begin{align}\label{4.20HZ}
\|u_{n}\|_{L^{\infty}(B_{R}(\tilde{y}_{n}))}\geq \gamma.
\end{align}
Indeed, if \eqref{4.20HZ} is not true, we get $\|u_{n}\|_{L^{\infty}(\R^{3})}< \gamma$, and it follows from $J_{\e_{n}}'(u_{n})=0$, \eqref{4.18HZ} and Lemma \ref{DI} that 
\begin{align*}
[|u_{n}|]^{2}+\int_{\R^{3}}V_{0}|u_{n}|^{2}dx&\leq \|u_{n}\|^{2}_{\e_{n}}+\int_{\R^{3}} \phi_{|u_{n}|}^{t}|u_{n}|^{2}dx\\
&=\int_{\R^{3}} g_{\e_{n}}(x, |u_{n}|^{2})|u_{n}|^{2}\,dx\\
&\leq \frac{V_{0}}{2}\int_{\R^{3}}|u_{n}|^{2}\, dx
\end{align*}
which implies that $\||u_{n}|\|_{H^{s}(\R^{3})}=0$, that is a contradiction. Then \eqref{4.20HZ} holds.\\
Using \eqref{4.19HZ} and \eqref{4.20HZ}, we can infer that the maximum points $p_{n}$ of $|u_{n}|$ belong to $B_{R}(\tilde{y}_{n})$, that is $p_{n}=\tilde{y}_{n}+q_{n}$ for some $q_{n}\in B_{R}$. Recalling that the associated solution of \eqref{P} is of the form $\hat{u}_{n}(x)=u_{n}(x/\e_{n})$, we can see that a maximum point $\eta_{\e_{n}}$ of $|\hat{u}_{n}|$ is $\eta_{\e_{n}}=\e_{n}\tilde{y}_{n}+\e_{n}q_{n}$. Since $q_{n}\in B_{R}$, $\e_{n}\tilde{y}_{n}\rightarrow y_{0}$ and $V(y_{0})=V_{0}$, from the continuity of $V$ we can conclude that
$$
\lim_{n\rightarrow \infty} V(\eta_{\e_{n}})=V_{0}.
$$
Finally, we give an estimate on the decay of $|\hat{u}_{n}|$.
Invoking Lemma $4.3$ in \cite{FQT}, we can find a function $w$ such that 
\begin{align}\label{HZ1}
0<w(x)\leq \frac{C}{1+|x|^{3+2s}},
\end{align}
and
\begin{align}\label{HZ2}
(-\Delta)^{s} w+\frac{V_{0}}{2}w\geq 0 \mbox{ in } \R^{3}\setminus B_{R_{1}} 
\end{align}
for some suitable $R_{1}>0$. Using Lemma \ref{moser}, we know that $v_{n}(x)\rightarrow 0$ as $|x|\rightarrow \infty$ uniformly in $n\in \mathbb{N}$, so there exists $R_{2}>0$ such that
\begin{equation}\label{hzero}
h_{n}=g(\e_{n}x+\e_{n}\tilde{y}_{n}, v_{n}^{2})v_{n}\leq \frac{V_{0}}{2}v_{n}  \mbox{ in } B_{R_{2}}^{c}.
\end{equation}
Let us denote by $w_{n}$ the unique solution to 
$$
(-\Delta)^{s}w_{n}+V_{0}w_{n}=h_{n} \mbox{ in } \R^{3}.
$$
Then $w_{n}(x)\rightarrow 0$ as $|x|\rightarrow \infty$ uniformly in $n\in \mathbb{N}$, and by comparison $0\leq v_{n}\leq w_{n}$ in $\R^{3}$. Moreover, in view of \eqref{hzero} and $\phi_{|w_{n}|}^{t}\geq 0$, it holds
\begin{align*}
(-\Delta)^{s}w_{n}+\frac{V_{0}}{2}w_{n}\leq h_{n}-\frac{V_{0}}{2}w_{n}\leq 0 \mbox{ in } B_{R_{2}}^{c}.
\end{align*}
Choose $R_{3}=\max\{R_{1}, R_{2}\}$ and we set 
\begin{align}\label{HZ4}
c=\inf_{B_{R_{3}}} w>0 \mbox{ and } \tilde{w}_{n}=(b+1)w-c w_{n}.
\end{align}
where $b=\sup_{n\in \mathbb{N}} \|w_{n}\|_{L^{\infty}(\R^{3})}<\infty$. 
Our goal is to show that 
\begin{equation}\label{HZ5}
\tilde{w}_{n}\geq 0 \mbox{ in } \R^{3}.
\end{equation}
Firstly, we observe that
\begin{align}
&\lim_{|x|\rightarrow \infty} \sup_{n\in \mathbb{N}}\tilde{w}_{n}(x)=0,  \label{HZ0N} \\
&\tilde{w}_{n}\geq bc+w-bc>0 \mbox{ in } B_{R_{3}} \label{HZ0},\\
&(-\Delta)^{s} \tilde{w}_{n}+\frac{V_{0}}{2}\tilde{w}_{n}\geq 0 \mbox{ in } \R^{3}\setminus B_{R_{3}} \label{HZ00}.
\end{align}
Now, we argue by contradiction, and we assume that there exists a sequence $(\bar{x}_{j, n})\subset \R^{3}$ such that 
\begin{align}\label{HZ6}
\inf_{x\in \R^{3}} \tilde{w}_{n}(x)=\lim_{j\rightarrow \infty} \tilde{w}_{n}(\bar{x}_{j, n})<0. 
\end{align}
From (\ref{HZ0N}), we can deduce that $(\bar{x}_{j, n})$ is bounded, and, up to subsequence, we may assume that there exists $\bar{x}_{n}\in \R^{3}$ such that $\bar{x}_{j, n}\rightarrow \bar{x}_{n}$ as $j\rightarrow \infty$. 
Thus, (\ref{HZ6}) yields
\begin{align}\label{HZ7}
\inf_{x\in \R^{3}} \tilde{w}_{n}(x)= \tilde{w}_{n}(\bar{x}_{n})<0.
\end{align}
Using the minimality of $\bar{x}_{n}$ and the representation formula for the fractional Laplacian (see Lemma 3.2 in \cite{DPV}), we can see that 
\begin{align}\label{HZ8}
(-\Delta)^{s}\tilde{w}_{n}(\bar{x}_{n})=\frac{c_{3, s}}{2} \int_{\R^{3}} \frac{2\tilde{w}_{n}(\bar{x}_{n})-\tilde{w}_{n}(\bar{x}_{n}+\xi)-\tilde{w}_{n}(\bar{x}_{n}-\xi)}{|\xi|^{3+2s}} d\xi\leq 0.
\end{align}
Taking into account (\ref{HZ0}) and (\ref{HZ6}), we obtain that $\bar{x}_{n}\in \R^{3}\setminus B_{R_{3}}$.
This together with (\ref{HZ7}) and (\ref{HZ8}) imply 
$$
(-\Delta)^{s} \tilde{w}_{n}(\bar{x}_{n})+\frac{V_{0}}{2}\tilde{w}_{n}(\bar{x}_{n})<0,
$$
which contradicts (\ref{HZ00}).
Thus (\ref{HZ5}) holds, and using (\ref{HZ1}) and $v_{n}\leq w_{n}$ we get
\begin{align*}
0\leq v_{n}(x)\leq w_{n}(x)\leq \frac{(b+1)}{c}w(x)\leq \frac{\tilde{C}}{1+|x|^{3+2s}} \mbox{ for all } n\in \mathbb{N}, x\in \R^{3},
\end{align*}
for some constant $\tilde{C}>0$. 
Therefore, recalling the definition of $v_{n}$, we can see that  
\begin{align*}
|\hat{u}_{n}|(x)&=|u_{n}|\left(\frac{x}{\e_{n}}\right)=v_{n}\left(\frac{x}{\e_{n}}-\tilde{y}_{n}\right) \\
&\leq \frac{\tilde{C}}{1+|\frac{x}{\e_{n}}-\tilde{y}_{n}|^{3+2s}} \\
&=\frac{\tilde{C} \e_{n}^{3+2s}}{\e_{n}^{3+2s}+|x- \e_{n} \tilde{y}_{n}|^{3+2s}} \\
&\leq \frac{\tilde{C} \e_{n}^{3+2s}}{\e_{n}^{3+2s}+|x-\eta_{\e_{n}}|^{3+2s}}.
\end{align*}
This ends the proof of Theorem \ref{thm1}.
\end{proof}

\smallskip
\noindent {\bf Acknowledgements.}
The author would like to thank the anonymous referee for her/his careful reading of the manuscript and valuable  suggestions that improved the presentation of the paper.

\end{document}